\documentclass[11pt]{article}
\usepackage{t1enc}
\usepackage{amsmath,amssymb,color,url}

\newtheorem{Thm}{Theorem}
\newtheorem{Def}{Definition}
\newtheorem{Prop}{Proposition}

\newtheorem{Lemma}{Lemma}
\newtheorem{Coro}{Corollary}
\newtheorem{Remark}{Remark}
\newtheorem{Example}{Example}

\newenvironment{proof}[1][Proof]{\textbf{#1.} }{\hfill $\square$}

\newcommand{\cad}{c\`adl\`ag }
\newcommand{\what}{\widehat}
\newcommand{\wtil}{\widetilde}
\newcommand{\dist}{\mbox{{\rm dist}}}

\def \vth{\vartheta}
\def \eps{\varepsilon}
\def \N{\mathbb{N}}
\def \R{\mathbb{R}}

\def \E{\mathbb{E}}
\def \F{\mathcal{F}}
\def \bF{\mathbb{F}}
\def \cY{\mathcal{Y}}
\def \bH{\mathbb{H}}

\def \bS{\mathbb{S}}
\def \tOm{\widetilde{\Omega}}

\def \cU{\mathcal{U}}
\def \cP{\mathcal{P}}
\def \cS{\mathcal{S}}
\def \cA{\mathcal{A}}
\def \cZ{\mathcal{Z}}

\def \cD{\mathcal{D}}
\def \cM{\mathcal{M}}
\def \ds{\displaystyle}
\def \tP{\widetilde{\mathcal{P}}}
\def \tpi{\widetilde{\pi}}
\def \bX{\overline{X}}
\def \ba{\overline{\alpha}}
\def \bb{\overline{\beta}}
\def \al{\alpha}
\def \P{\mathbb{P}}
\def \1{\mathbf{1}}

\newcommand*{\Id}{\operatorname{Id}}
\newcommand*{\trace}{\operatorname{Trace}}
\DeclareMathOperator*{\essinf}{\operatorname{essinf}}

\begin{document}

\title{Minimal supersolutions for BSDEs with singular terminal condition and application to optimal position targeting.}
\author{T. Kruse \thanks{University of Duisburg-Essen, Thea-Leymann-Str. 9, 45127 Essen, Germany,
e-mail: {\tt thomas.kruse@uni-due.de}
}
, A. Popier \thanks{Laboratoire Manceau de Math\'ematiques, Universit\'e du
Maine, Avenue Olivier Messiaen, 72085 Le Mans, Cedex 9, France.\hfill\break
e-mail: {\tt alexandre.popier@univ-lemans.fr}
}
}
\date{\today}

\maketitle

\begin{abstract}
We study the existence of a minimal supersolution for backward stochastic differential equations when the terminal data can take the value $+\infty$ with positive probability. We deal with equations on a general filtered probability space and with generators satisfying a general monotonicity assumption. With this minimal supersolution we then solve an optimal stochastic control problem related to portfolio liquidation problems. We generalize the existing results in three directions: firstly there is no assumption on the underlying filtration (except completeness and quasi-left continuity), secondly we relax the terminal liquidation constraint and finally the time horizon can be random.
\end{abstract}

\section*{Introduction}

This paper is devoted to the study of backward stochastic differential equations (BSDEs) with {\it singular} terminal condition. We adopt from \cite{popi:06} and \cite{popi:07} the notion of a weak (super)solution $(Y,\psi,M)$ to a BSDE of the following form
\begin{equation} \label{eq:bsde_2}
dY_t  = - f(t,Y_t,\psi_t) dt + \int_\cZ \psi_t(z) \tpi(dz,dt) + d M_t,
\end{equation}
where $\tpi$ is a compensated Poisson random measure on a probability space $(\Omega,\F,\P)$ with a filtration $\bF = (\F_t)_{t\geq 0}$. The filtration $\bF$ is supposed to be complete and right continuous. In particular, it can support a Brownian motion orthogonal to $\tpi$. The solution component $M$ is required to be a local martingale orthogonal to $\tpi$. The function $f:\Omega\times \R_+ \times \R \times \R^d\to \R$ is called the {\it driver} (or {\it generator}) of the BSDE. The particularity here is that we allow the {\it terminal condition} $\xi$ to be {\it singular}: for a stopping time $\tau$, the random variable $\xi$ is $\F_\tau$-measurable and takes the value $+\infty$ with positive probability. 

In our first main result (Theorem \ref{thm:main_thm_1}) we establish existence of a {\it minimal} weak supersolution to \eqref{eq:bsde_2}. This supersolution is constructed via approximation from below. For each $L>0$ we consider a truncated version of \eqref{eq:bsde_2} with terminal condition $\xi\wedge L$. We impose that the driver $f$ satisfies a monotonicity assumption in the $y$-variable and is Lipschitz continuous with respect to $\psi$. Then existence, uniqueness and comparison results for a solution $(Y^L,\psi^L,M^L)$ to the truncated BSDE can be deduced from \cite{krus:popi:14}, where the theory of BSDEs with a monotone driver in a general filtration has been developed. We obtain the minimal supersolution $(Y,\psi,M)$ with singular terminal condition by passing to the limit $L\to \infty$. The crucial task is to establish suitable a priori estimates for $Y^L$ guaranteeing that when passing to the limit the solution $Y$ does {\it not} explode before time $\tau$. To this end, the generator $f$ cannot be Lipschitz continuous w.r.t. $y$. Hence we impose that $f$ is monotone and decreases at least polynomially with random coefficient in the $y$-variable. In the case where $\tau$ is deterministic this condition suffices to ensure boundedness of $Y^L$. When $\tau$ is random, we restrict attention to first exit of diffusions from a regular set.

BSDEs with singular terminal condition were already studied in \cite{anki:jean:krus:13} and \cite{popi:06} for deterministic terminal time (see also \cite{grae:hors:qiu:13} for a treatise on BSPDEs), and in \cite{popi:07} for a random terminal time. Let us briefly outline in which directions our findings generalize some results from these papers.
\begin{itemize}
\item {\it General driver} $f$. Indeed, in the previously mentioned papers $f$ is assumed to be a polynomial function of $y$ (plus possibly a particular bounded from above function of $\psi$ in \cite{grae:hors:qiu:13}). Here $f$ is supposed to be only bounded from above by a polynomial function w.r.t. $y$. The fact that we only assume here that $f$ is Lipschitz continuous with respect to $\psi$ but not necessarily bounded, requires to derive new a priori estimates for the family of solutions $(Y^L)$. Moreover as in \cite{anki:jean:krus:13}, the generator can be {\it singular} in the sense that the process $f^0_t = f(t,0,0)$ can explode at time $\tau$. We only impose an integrability condition on $f^0$ which is weaker than the condition in \cite{anki:jean:krus:13}. This weaker integrability condition and the occurence of jumps imply that the convergence of the approximating sequence $(Y^L)_{L>0}$ has to be handled more carefully (see in particular the proof of Proposition \ref{thm:exists_sing_sol} where technical details are postponed in the appendix). BSDEs where the generator possesses a singularity in the time variable were studied in \cite{jean:reve:14} and \cite{jean:mast:poss:15} to solve utility maximization problems with random horizon.

\item {\it General filtration} $\bF$. Moreover, compared to the papers \cite{anki:jean:krus:13}, \cite{popi:06} and \cite{popi:07}, we do not restrict attention to a filtration generated by Brownian and Poisson noise. Here the filtration $\bF$ satisfies only the standard assumptions (completeness and right-continuity). Hence the additional local martingale part $M$ appears in the BSDE and has to be controlled when we let $L$ go to $+\infty$. The quasi left-continuity condition on $\bF$ will be imposed only to ensure the lower semi-continuity of $Y$ at time $\tau$: $\displaystyle \liminf_{t\to \tau} Y_{t}\geq \xi$.

\item {\it Random terminal time} $\tau$. To our best knowledge, \cite{popi:07} is the only paper that deals with a singular terminal condition at a random time $\tau$. In this work, the generator $f$ is equal to $f(y) = - y |y|^{q-1}$ for some $q > 1$ and the filtration is generated by a Brownian motion. When the terminal time is random, the derivation of the a priori estimate for the sequence $Y^L$ is more involved than in the deterministic case. For a general random time $\tau$, we show that the limit process $Y$ may be infinite before time $\tau$. For this reason, we consider the first exit time of a continuous diffusion from a regular set and our estimate is a generalization of the Keller-Osserman inequality. 
\end{itemize}
We also note that our results can be extended to the case where the driver is additionally a Lipschitz continuous function of a variable $Z$, which represents the integrand in the martingale representation w.r.t.\ a Brownian motion (c.f.\ Remark \ref{rem:Brownian_motion}). 

\vspace{1cm}
Since the seminal paper by Pardoux and Peng \cite{pard:peng:90} BSDEs have proved to be a powerful tool to solve stochastic optimal control problems (see e.g.\ the survey article \cite{el1997backward} or the book \cite{pham2009continuous}). In the second part of the paper we use the notion of weak supersolutions to provide a purely probabilistic solution of a stochastic control problem with a terminal constraint on the controlled process. More precisely, we consider the problem of minimizing the cost functional \footnote{We define $0\cdot \infty := 0$.}
\begin{equation}\label{eq:control_pb_intro}
J(X) = \E \left[  \int_0^\tau \left( \eta_s |\alpha_s|^p + \gamma_s |X_s|^p + \int_\cZ \lambda_s(z) |\beta_s(z)|^p \mu(dz) \right) ds  + \xi  |X_\tau|^p  \right]
\end{equation}
over all progressively measurable processes $X$ that satisfy the dynamics
\begin{equation*}
X_s =x +\int_0^s \alpha_u du +\int_0^s \int_\cZ \beta_u(z) \pi(dz,du).
\end{equation*}
Here $p>1$ and the processes $\eta, \gamma$ and $\lambda$ are non negative progressively measurable. Again the $\mathcal F_\tau$-measurable random variable $\xi$ takes the value $\infty$ with positive probability. This singularity imposes the terminal state constraint on the set of strategies. Indeed, any strategy $X$ that does not satisfy this terminal constraint creates infinite costs. In particular, such a strategy cannot be optimal if there exists some strategy that creates finite costs (which will always be the case under the assumptions that we impose). In the cases where $\tau$ is deterministic or a first exit time, we characterize optimal strategies and the value function of this control problem with the BSDE
\begin{equation} \label{eq:bsde_contr_prob}
dY_t  = (p-1)  \frac{Y_t^{q}}{\eta_t^{q-1}} dt + \Theta(t,Y_t,\psi_t) dt - \gamma_t dt + \int_\cZ \psi_t(z) \tpi(dz,dt) + d M_t
\end{equation}
with $\displaystyle \liminf_{t\to \tau} Y_t \geq \xi$. Here $q > 1$ is the H\"older conjugate of $p$ and $\Theta$ is a Lipschitz continuous function (see \eqref{eq:generator} for the precise definition). We provide sufficient conditions on the coefficient processes $\eta, \gamma$ and $\lambda$ such that Theorem \ref{thm:main_thm_1} ensures existence of a minimal weak supersolution to \eqref{eq:bsde_contr_prob} and carry out a verification that is based on a penalization argument.

The analysis of optimal control problems with state constraints on the terminal value is motivated by models of optimal portfolio liquidation under stochastic price impact. The traditional assumption that all trades can be settled without impact on market dynamics is not always appropriate when investors need to close large positions over short time periods. In recent years models of optimal portfolio liquidation have been widely developed, see, e.g. \cite{almg:12}, \cite{almg:chri:01}, \cite{fors:kenn:12}, \cite{gath:schi:11}, \cite{hors:nauj:14}, or \cite{krat:scho:13}, among many others.

Variants of the position targeting problem \eqref{eq:control_pb_intro} have been studied in \cite{anki:jean:krus:13}, \cite{anki:krus:13b}, \cite{schie:13}, \cite{grae:hors:qiu:13} or \cite{grae:hors:sere:13}. In this framework the state process $X$ denotes the agent's position in the financial market. She has two means to control her position. At each point in time $t$ she can trade in the primary venue at a rate $\alpha_t$ which generates costs $\eta_t |\alpha_t|^p$ incurred by the stochastic price impact parameter $\eta_t$. Moreover, she can submit passive orders to a secondary venue ("dark pool"). These orders get executed at the jump times of the Poisson random measure $\pi$ and generate so called slippage costs $\int_\cZ \lambda_t(z) |\beta_t(z)|^p\mu(dz)$. We refer to \cite{krat:scho:13} for a more detailed discussion. The term $\gamma_t |X_t|^p$ can be understood as a measure of risk associated to the open position. $J(X)$ thus represents the overall expected costs for closing an initial position $x$ over the time period $[0,\tau]$ using strategy $X$.

Our approach allows to incorporate some novel features into optimal liquidation models. First, we do not impose any assumption on the filtration (except quasi-left continuity). For the financial model, this means that the noise is not necessarily generated by a Brownian motion. Moreover, the liquidation constraint is relaxed in the following way. Instead of enforcing the condition $X_\tau =0$ a.s., that is the position has to be closed imperatively, our model is flexible enough to allow for a specification of a set of market scenarios $\cS \subset \F_\tau$ where liquidation is mandatory: $X_\tau \1_{\cS} = 0$. On the complement $\cS^c$ a penalization depending on the remaining position size can be implemented. This terminal constraint is described by the $\F_\tau$-measurable non negative random variable $\xi$ such that $\cS = \{ \xi = +\infty\}$. Thus for a binding liquidation $X_\tau=0$, we take $\xi = +\infty$ a.s. For excepted scenarios, we can consider $\xi = \infty \1_\cS$ with for example $\cS = \{ \max_{t \in [0,T]} \eta_t \leq H\}$ or $\cS = \{ \int_0^T \eta_t dt \leq H\}$ for a given threshold $H>0$. This means that liquidation is only mandatory if the maximal price impact (or the average price impact) is small enough throughout the liquidation period. If the illiquidity of the market is too high, the trader has not obligatorily to close his position.
Finally, our model allows for a random time horizon $\tau$. For example, one can consider {\it price-sensitive} liquidation periods where the position has to be closed before the first time when the unaffected market price $S$ (a diffusion) falls below some threshold level $K>0$, i.e. $\tau=\inf\{t\ge 0|S_t\le K\}$.

The paper is decomposed as follows. In the first section, we give the mathematical setting and present the main results concerning the BSDE \eqref{eq:bsde_2}. The set of assumptions will differ in the two cases $\tau$ deterministic (Theorem \ref{thm:main_thm_1}) and $\tau$ random (Theorem \ref{thm:main_thm_2}). We construct a supersolution of the BSDE \eqref{eq:bsde_2} using truncation arguments as in \cite{popi:06} or \cite{anki:jean:krus:13} and we prove that this solution is minimal. As mentioned before the main difficulties are to control the sequence of solutions for the truncated BSDE (see Propositions \ref{prop:upper_bound_Y_L} and \ref{prop:Keller-Osserman}) and to prove the convergence of the approximating sequence. In Section \ref{sect:back_control} we use the previous results to obtain a minimal supersolution for BSDE \eqref{eq:bsde_contr_prob} and we verify that this solution gives the value function and an optimal control for the optimal position targeting problem  (Theorem \ref{thm:main_thm_3}).

\section{Minimal supersolutions for the singular BSDE} \label{sect:exist_min_sol}

\subsection{Setting and notation}\label{setting and notation}

We consider a filtered probability space $(\Omega,\F,\P,\bF = (\F_t)_{t\geq 0})$. The filtration is assumed to be complete and right continuous. Moreover, we assume that $\bF$ is quasi-left continuous, which means that for every sequence $(\tau_n)$ of $\bF$ stopping times such that $\tau_n \nearrow \tilde \tau$ for some stopping time $\tilde \tau$ we have $\bigvee_{n\in \N}\F_{\tau_n}=\F_{\tilde \tau}$. We assume that $(\Omega,\F,\P,\bF = (\F_t)_{t\geq 0})$ supports a Poisson random measure $\pi$ with intensity $\mu(dz)dt$ on the space $\cZ \subset \R^d \setminus \{0\}$. The measure $\mu$ is $\sigma$-finite on $\cZ$ such that
$$\int_\cZ (1\wedge |z|^2) \mu(dz) <+\infty.$$

By $\cP$ we denote the predictable $\sigma$-field on $\Omega \times \R_+$. We set $\tP=\cP \otimes \mathcal{B}(\cZ)$ where $\mathcal{B}(\cZ)$ is the Borelian $\sigma$-field on $\cZ$. On $\tOm = \Omega \times [0,T] \times \cZ$, a function that is $\tP$-measurable, is called predictable. $G_{loc}(\pi)$ is the set of $\tP$-measurable functions $\psi$ on $\tOm$ such that for any $t \geq 0$ a.s.
$$ \int_0^t \int_\cZ (|\psi_s(z)|^2\wedge |\psi_s(z)|) \mu(dz) ds < +\infty.$$
For any stopping time $\tilde \tau$ and $m> 1$, the set $L^m_\pi(0,\tilde \tau)$ contains all processes $\psi \in G_{loc}(\mu)$ such that
$$\E \left[ \int_0^{\tilde \tau} \int_{\cZ} |\psi_s(z)|^m \mu(dz) ds  \right] < +\infty .$$
By $L^m_\mu=L^m(\cZ,\mu;\R^d)$ we denote the set of measurable functions $\psi : \cZ \to \R^d$ such that
$$\| \psi \|^m_{L^m_\mu} = \int_{\cZ} |\psi(z)|^m \mu(dz)  < +\infty .$$

By $ \cM^\perp$ we denote the set of c\`adl\`ag local martingales orthogonal to $\tpi$. If $M \in \cM^\perp$ then
$\E(\Delta M * \pi | \tP) = 0,$ where the product $*$ denotes the integral process (see II.1.5 in \cite{jaco:shir:03}).
For any stopping time $\tilde \tau$ the set $\cM^m(0,\tilde \tau)$ is the subset of all martingales such that
$\E \left( [M]_{\tilde \tau}^{m/2}\right) < +\infty.$
Finally, for $m> 1$, $\bS^m(0,\tilde \tau)$ is the set of all progressively measurable \cad processes $F$ such that $\E \left[  \sup_{t\in [0,\tilde \tau]} |F_t|^m \right] < +\infty.$ The set $\bH^m(0,\tilde \tau)$ contains all progressively measurable \cad processes $F$ such that $\E \left[ \left( \int_{0}^{\tilde \tau} |F_t|^2 dt \right)^{m/2} \right] < +\infty.$

\subsection{Deterministic terminal times}\label{bsde_det_time}

In this section let $T>0$ and let $\xi$ be a $\F_T$-measurable random variable. We denote by $\cS$ the set $\{\xi = +\infty\}$. 
Since we explicitly allow $\xi$ to take the value $+\infty$ with positive probability, we need to specify a weak notion of solutions to \eqref{eq:bsde_2}. We relax the usual definition of a solution to a BSDE by only requiring that \eqref{eq:bsde_2} holds strictly before time $T$.

\begin{Def}[Weak supersolution in the case of deterministic terminal times] \label{def:sol_sing_BSDE}
We say that a triple of processes $(Y,\psi,M)$ is a supersolution to the BSDE \eqref{eq:bsde_2} with singular terminal condition $Y_T = \xi$ if it satisfies:
\begin{enumerate}
\item $M\in \cM^\perp$, $\psi \in G_{loc}(\pi)$ and there exists some $\ell > 1$ such that for all $t<T$:
$$\E \left( \sup_{s\in [0,t]} |Y_{s}|^\ell + \int_0^{t} \int_\cZ |\psi_s(z)|^\ell \mu(dz) ds + [M ]^{\ell/2}_{t} \right) < +\infty;$$
\item $Y$ is bounded from below by a process $\bar Y \in \bS^2(0,T)$;
\item for all $0\leq s \leq t<T$:
\begin{eqnarray*}
Y_{s}  =  Y_{t } + \int_{s }^{t}  f(u,Y_u,\psi_u)du -  \int_{s }^{t} \int_\cZ \psi_u(z) \tpi(dz,du) - \int_{s}^{t } dM_u.
\end{eqnarray*}
\item  $ \ds \liminf_{t \to T} Y_{s } \geq \xi$ a.s.
\end{enumerate}
We say that $(Y,\psi,M)$ is a minimal supersolution to the BSDE \eqref{eq:bsde_2} if for any other supersolution $(Y',\psi',M')$ we have $Y_t\le Y'_t$ a.s.\ for any $t\in [0,T)$.
\end{Def}

To establish existence of a minimal supersolution to BSDE \eqref{eq:bsde_2} we impose the following conditions on the driver $f : \Omega \times [0,T] \times \R \times \R^d \to \R$.  For notational convenience we write $f^0_t=f(t,0,0)$.
\begin{description}
\item[A1.] The function $y\mapsto f(t,y,\psi)$ is continuous and monotone: there exists $\chi \in \R$ such that a.s. and for any $t \in [0,T]$ and $\psi \in L^2_\mu$
\begin{equation*}
(f(t,y,\psi)-f(t,y',\psi))(y-y') \leq \chi (y-y')^2.
\end{equation*}
\item[A2.] There exists a progressively measurable process $\kappa = \kappa^{y,\psi,\phi} : \Omega \times \R_+ \times \cZ \to \R$ such that
\begin{equation*}
f(t,y,\psi)-f(t,y,\phi) \leq \int_\cZ (\psi(z)-\phi(z))  \kappa^{y,\psi,\phi}_t(z)  \mu(dz)
\end{equation*}
with $\P\otimes Leb \otimes \mu$-a.e. for any $(y,\psi,\phi)$, $-1 \leq \kappa^{y,\psi,\phi}_t(z)$
and $|\kappa^{y,\psi,\phi}_t(z)| \leq \vartheta(z)$ where $\vartheta \in L^2_\mu$.
\item[A3.] For every $n> 0$ it holds that 
$\sup_{|y|\leq n} |f(t,y,0)-f^0_t| \in L^1((0,T)\times \Omega).$
\item[A4.] The negative parts of $\xi$ and $f^0$ are square integrable: $\xi^- \in L^2(\Omega) \text{ and }(f^0)^- \in L^2((0,T)\times \Omega).$
\end{description}

Conditions \textbf{A1} to \textbf{A4} will ensure existence and uniqueness of the solution for a version of BSDE \eqref{eq:bsde_2}, where the terminal condition $\xi$ is replaced by $\xi \wedge L$ and the generator $f$ by $f^L$ (see \eqref{eq:generator_trunc_BSDE}) for some $L>0$. We obtain the minimal supersolution with singular terminal condition $\xi$ by letting $L$ tend to $\infty$. To ensure that in the limit $L\to \infty$ the solution component $Y$ attains the value $\infty$ on $\cS$ at time $T$ but is finite before time $T$, we have to impose some further growth behavior on $f$. We assume that $f$ decreases at least polynomially in the $y$-variable.
\begin{description}
\item[A5.] There exists a constant $q > 1$ and a positive process $\eta$ such that for any $y \geq 0$
\begin{equation*}
f(t,y,\psi)\leq -\frac{p-1}{\eta^{q-1}_t}|y|^{q} + f(t,0,\psi).
\end{equation*}
$p$ is the H\"older conjugate of $q$. 
\item[A6.] There exists $\ell>1$ such that
$\E \int_0^T \left[\eta_s + (T-s)^p (f^0_s)^+\right]^{\ell} ds < +\infty.$
\item[A7.] There exists $k > \max(\frac{\ell}{\ell-1} , 2)$ such that 
$\int_\cZ |\vartheta(z)|^{k} \mu(dz) < +\infty.$
\end{description}

\noindent \textbf{Assumptions (A).} We say that Assumptions $\textbf{(A)}$ are satisfied if all seven hypotheses \textbf{A1} to \textbf{A7} hold.
\hfill $\diamond$

\begin{Remark}[on A1] \label{rem:monotone_coeff}
We can suppose w.l.o.g.\ that $\chi=0$. Indeed if $(Y,\psi,M)$ is a solution of \eqref{eq:bsde_2} then $(\bar Y,\bar \psi, \bar M)$ with
$$\bar Y_t = e^{\chi t} Y_t, \quad \bar \psi_t  = e^{\chi t} \psi_t, \quad d \bar M_t = e^{\chi t} dM_t $$
satisfies an analogous BSDE with terminal condition $\bar \xi = e^{\chi T}\xi$, and generator
\begin{eqnarray*}
\bar f(t,y,\psi) & = & \left[ e^{\chi t} f(t,e^{-\chi t} y, e^{-\chi t} \psi) - \chi y \right]  
\end{eqnarray*}
and $\bar f$ satisfies the same assumptions with $\chi = 0$. \textbf{\emph{In the rest of this section, we will suppose that $\chi = 0$}}.
\end{Remark}
\begin{Remark}[on A2]\label{rem:f_Lip}
The second condition \emph{\textbf{A2}} implies that $f$ is Lipschitz continuous w.r.t. $\psi$ uniformly in $\omega$, $t$ and $y$. Indeed by Cauchy-Schwarz's inequality
$$f(t,y,\psi) - f(t,y,\phi)  \leq \left|\int_\cZ (\psi(z)-\phi(z))  \kappa^{y,\psi,\phi}_t(z)  \mu(dz) \right| \leq  \|\vartheta\|_{L^2_\mu}  \|\psi-\phi\|_{L^2_\mu}.$$
And conversely since
$$f(t,y,\phi) - f(t,y,\psi)  \leq \int_\cZ (\phi(z)-\psi(z))  \kappa^{y,\phi,\psi}_t(z)  \mu(dz),$$
we obtain
$$f(t,y,\psi) - f(t,y,\phi)  \leq  \|\vartheta\|_{L^2_\mu}  \|\psi-\phi\|_{L^2_\mu}.$$
\end{Remark}
\begin{Remark}[on A5]
It follows from Condition \emph{\textbf{A3}} and \emph{\textbf{A5}} that the process $1/\eta^{q-1}$ must be in $L^1((0,T)\times \Omega)$ 
\begin{equation} \label{eq:cond_1_over_alpha}
\E \int_0^T \frac{1}{\eta^{q-1}_t} dt < +\infty.
\end{equation}
Let us just mention that it is possible to assume only integrability w.r.t. $t$ a.s. in \emph{\textbf{A2}} (see \cite{bria:dely:hu:03}, remark 4.3).
\end{Remark}

In this section, our main result can be summarized as follows.
\begin{Thm}\label{thm:main_thm_1}
Under Assumptions \emph{\textbf{(A)}} there exists a minimal supersolution $(Y,\psi,M)$ to \eqref{eq:bsde_2} with singular terminal condition $Y_T=\xi$.
\end{Thm}
To prove Theorem \ref{thm:main_thm_1} we proceed as in \cite{anki:jean:krus:13} by truncation. The complete statement and the proof of this result is divided into Propositions \ref{prop:exist_solution_trunc_BSDE}, \ref{prop:upper_bound_Y_L}, \ref{thm:exists_sing_sol} and \ref{prop:minimality}. For any $L \geq 0$ we consider the BSDE
\begin{equation} \label{eq:truncated_bsde}
dY^L_t = -  f^L(t,Y^L_t,\psi^L_t) dt + \int_\cZ \psi^L_t(z) \tpi(dz,dt) +dM^L_t
\end{equation}
with bounded terminal condition $Y^L_T = \xi \wedge L$ and where
\begin{equation}\label{eq:generator_trunc_BSDE}
f^L(t,y,\psi) = ( f(t,y,\psi)-f^0_t ) + f^0_t \wedge L.
\end{equation}
\begin{Prop} \label{prop:exist_solution_trunc_BSDE}
Under Assumptions \emph{\textbf{(A)}}, there exists for every $L>0$ a unique solution $(Y^L,\psi^L,M^L)$ to \eqref{eq:truncated_bsde} with $Y^L\in \bS^2(0,T)$, $\psi^L \in L^2_\pi(0,T)$, $M^L \in \cM^2(0,T)\cap \cM^\perp$.
Moreover there exists a process $\bar Y$ in $\bS^2(0,T)$, independent of $L$, such that a.s. for any $t \in[0,T]$, $\bar Y_t \leq Y^L_t$.  If $(f^0_t)^- = \xi^-=0$, then $\bar Y_t =0$, and $Y^L_t$ is non negative.
\end{Prop}

\begin{proof}
From assumptions \textbf{A1}, \textbf{A2} and \textbf{A4}, it follows that $f^L$ is monotone w.r.t. $y$, Lipschitz continuous w.r.t. $\psi$, and $f^L(t,0,0) = f^0_t \wedge L \in L^2((0,T)\times \Omega)$. Moreover for every $n > 0$ and $|y|\leq n$:
\begin{eqnarray*}
|f^L(t,y,0) - f^L(t,0,0)| & = & |f(t,y,0)-f^0_t| \leq \sup_{|y|\leq n} |f(t,y,0)-f^0_t|.
\end{eqnarray*}
By Assumption \textbf{A3}, the mapping $t\mapsto \sup_{|y|\leq n} |f(t,y,0)-f^0_t|$ is in $L^1((0,T)\times \Omega)$. From Theorem 1 in \cite{krus:popi:14} it follows that there exists a unique solution $(Y^{L},\psi^{L},M^{L})$ to \eqref{eq:truncated_bsde} with terminal condition $\xi \wedge L$. This solution satisfies
$$\E \left[ \sup_{0\leq t \leq T} |Y^{L}_t|^2 + \int_0^T \int_\cZ (\psi^{L}_t(z))^2 \mu(dz) dt +[ M^{L}]_T \right] < +\infty.$$
Next, we construct the lower bound $\bar Y$. Let us take $\zeta= -\xi^-$ and $g(t,y,\psi) = ( f(t, y,\psi)-f^0_t) -(f^0_t)^- .$ The solution $(\bar Y,\bar \psi, \bar M)$ with $\bar Y \in \bS^2(0,T)$ of the BSDE with data $(\zeta,g)$ does not depend on $L$, and by comparison (Proposition 4 in \cite{krus:popi:14}) we have $\bar Y_t \leq Y^L_t$ a.s. for any $t \in [0,T]$. 
\end{proof}

Next, we derive an upper bound for the family $Y^L$ which is independent of $L$.

\begin{Prop}\label{prop:upper_bound_Y_L}
For every $t\in [0,T]$ the random variable $Y^L_t$ is bounded from above by $L(1+T)$ and for $t \in [0,T)$ the following estimate holds:
\begin{equation}\label{eq:a_priori_estimate_Y_L}
Y^{L}_t \leq \frac{K_{\vartheta}}{(T-t)^p} \left[ \E \left( \ \int_t^{T} \left( \eta_s + (T-s)^p (f^0_s)^+ \right)^{\ell} ds \bigg| \F_t\right) \right]^{1/\ell}
\end{equation}
where $K_{\vartheta}$ is a constant depending only on $\vartheta$.
\end{Prop}

\begin{proof}
Let us first consider the triple $(A_t,B_t,C_t) = (L(1+(T-t)),0,0)$. It is the solution of the BSDE with terminal condition $L$ and constant generator equal to $L$. By assumption \textbf{A1}, $f$ is monotone and hence it holds that $f(t,A_t,B_t) \leq f^0_t$. Thus by the definition \eqref{eq:generator_trunc_BSDE} of $f^L$ we have $f^L(t,A_t,B_t) \leq L.$
By the comparison principle (Proposition 4 in \cite{krus:popi:14}) we obtain $Y^{L}_t  \leq A_t \leq L(T+1)$ a.s. for any $t\in [0,T]$,

This upper bound depends on $L$. Next, we verify \eqref{eq:a_priori_estimate_Y_L}.
We consider the driver
\begin{eqnarray*}
h(t,y,\psi) & = &b^L_t - p\frac{1}{T-t} y + f(t,0,\psi).
\end{eqnarray*}
with $b^L_t =  \frac{\eta_t}{(T-t)^p} + ((f^0_t)^+ \wedge L)$.
Let $\eps > 0$ and denote by $(\cY^{\eps,L},\phi^{\eps,L},N^{\eps,L})$ the solution process of the BSDE on $[0,T-\eps]$ with driver $h$ and terminal condition $\cY^{\eps,L}_{T-\eps}=Y^{L,+}_{T-\eps} \geq 0$. Recall that
$$f(t,0,\psi) \leq \int_\cZ \psi(z) \kappa^{0,\psi,0}_t(z)  \mu(dz).$$
Hence by a comparison argument with the solution for linear BSDE (see \cite{quen:sule:13}, Lemma 4.1) we have
$$\cY^{\eps,L}_t \leq \E \left[ \Gamma_{t,T-\eps} Y^{L,+}_{T-\eps} + \int_t^{T-\eps} \Gamma_{t,s} b^L_s ds \bigg| \F_t\right]$$
where for $t \leq s \leq T-\eps$
$$\Gamma_{t,s} = \exp\left( -\int_t^s \frac{p}{T-u} du \right)V^{\eps,L}_{t,s} =\left( \frac{T-s}{T-t} \right)^{p} V^{\eps,L}_{t,s}$$
and
\begin{equation}\label{eq:auxil_proc_V}
V^{\eps,L}_{t,s} = 1+ \int_t^s \int_{\cZ}  V^{\eps,L}_{t,u^-} \kappa^{0,\phi^{\eps,L},0}_u(z) \tpi(dz,du).
\end{equation}
Hence
$$\cY^{\eps,L}_t \leq \frac{1}{(T-t)^p} \E \left[\eps^\rho V^{\eps,L}_{t,T-\eps} Y^{L,+}_{T-\eps} + \int_t^{T-\eps} V^{\eps,L}_{t,s} (T-s)^p b^L_s  ds \bigg| \F_t\right].$$
Since $b^L \geq 0$ it holds that $\cY^{\eps,L}_t  \geq 0$ a.s. for every $t \in [0,T]$. Hence from Condition \textbf{A5}
$$f^L(t,\cY^{\eps,L}_t,\phi^{\eps,L}_t) \leq - \frac{p-1}{\eta^{q-1}_t } (\cY^{\eps,L}_t)^{q}+ f^L(t,0,\phi^{\eps,L}_t).$$
It follows that
\begin{eqnarray*}
f^L(t,\cY^{\eps,L}_t,\phi^{\eps,L}_t) & \leq & h(t,\cY^{\eps,L}_t,\phi^{\eps,L}_t)-\frac{p-1}{\eta^{q-1}_t } (\cY^{\eps,L}_t)^{q} -\frac{a_t^{p-1}}{(T-t)^p}+\frac{p}{T-t} \cY^{\eps,L}_t \\
&\leq& h(t,\cY^{\eps,L}_t,\phi^{\eps,L}_t),
\end{eqnarray*}
where we used the Young inequality: $c^p+(p-1)y^q-p cy\ge 0$ which holds for all $c,y\ge 0$. The comparison theorem implies $Y^{L}_t \leq \cY^{\eps,L}_t$ for all $t\in[0,T-\eps]$ and $\eps > 0$.

Recall once again from Condition \textbf{A7} that $V^{\eps,L}_{t,.}$ belongs to $\bH^k(0,T-\eps)$ for $k\geq 2$.
From the upper bound $Y^{L}_t  \leq A_t \leq L(T+1)$ and from the integrability property of $V^{\eps,L}_{t,.}$, with dominated convergence, by letting $\eps \downarrow 0$ we obtain a.s.
$$\E \left[\eps^p V^{\eps,L}_{t,T-\eps} Y^{L,+}_{T-\eps}  \bigg| \F_t\right] \longrightarrow 0.$$
From Assumption \textbf{A7} and by the proof of Proposition A.1 in \cite{quen:sule:13}, there exists a constant $K_{\vartheta}$ such that a.s.
$$\E \left[\int_t^{T-\eps} (V^{\eps,L}_{t,s})^{k}  ds\bigg| \F_t\right] \leq K_{\vartheta}.$$
From Assumption \textbf{A6}, it follows that the process $((T-t)^p b^L_t, \ 0\leq t \leq T)$ belongs to $\bH^{\ell}(0,T)$. Therefore by H\"older inequality we obtain
\begin{eqnarray*}
&& \E \left[\int_t^{T-\eps} V^{\eps,L}_{t,s}(T-s)^p b^L_s  ds \bigg| \F_t\right] \leq K_\vartheta \E \left[\int_t^{T} ((T-s)^p b^L_s)^{\ell}  ds \bigg| \F_t\right]^{1/\ell}.
\end{eqnarray*}
Hence we can pass to the limit as $\eps \downarrow 0$
\begin{eqnarray*}
Y^{L}_t & \leq & \frac{K_\vartheta}{(T-t)^p} \E \left[ \ \int_t^{T}  ((T-s)^p b^L_s)^{\ell} ds \bigg| \F_t\right]^{1/\ell}.
\end{eqnarray*}
Assumption \textbf{A6} implies by monotone convergence for $L \to \infty$
\begin{eqnarray*}
Y^{L}_t & \leq &  \frac{K_\vartheta}{(T-t)^p} \E \left[ \ \int_t^{T} \left( \eta_s + (T-s)^p (f^0_s)^+ \right)^{\ell} ds \bigg| \F_t\right]^{1/\ell} <+\infty
\end{eqnarray*}
Thus we obtain the upper bound in \eqref{eq:a_priori_estimate_Y_L}.
\end{proof}

The constants $K_\vartheta$ and $\ell> 1$ in \eqref{eq:a_priori_estimate_Y_L} come from the growth condition on $f$ w.r.t. $\psi$ and from the lack of an estimate of $\psi^L$ independent of $L$. If we assume that $f(t,0,\psi)$ is bounded, then we can obtain a simpler estimate. 
\begin{Lemma} \label{lem:another_estimate}
If there exists a non negative process $K^f_t$ such that a.s. for any $t$ and $\psi$,
\begin{equation}\label{eq:f_growth_psi_2} 
f(t,0,\psi) \leq K^f_t, \quad  \mbox{with} \quad \E \int_0^T (T-s)^p K^f_s ds < +\infty
\end{equation}
then
\begin{equation}\label{eq:a_priori_estimate_Y_L_2}
Y^{L}_t  \leq \frac{1}{(T-t)^p} \E \left[ \ \int_t^{T} \left(\eta_s +2 (T-s)^p K^f_s \right) ds \bigg| \F_t\right].
\end{equation}
\end{Lemma}
\begin{proof}
The proof is almost the same as for Proposition \ref{prop:upper_bound_Y_L}. Therefore, we only outline the main modification. Note that \eqref{eq:f_growth_psi_2} implies that $f^0_t \leq K^f_t$ a.s. We consider the generator $h$ given by
\begin{equation*}
h(t,y,\psi) = h (t,y) = \frac{\eta_t}{(T-t)^p} + 2 K^f_t - p\frac{1}{T-t} y = b_t - p\frac{1}{T-t} y  .
\end{equation*}
Since $h$ is linear and does not depend on $\psi$, we have:
$$\cY^{\eps,L}_t = \frac{1}{(T-t)^p} \E \left[ \eps^p Y^{L,+}_{T-\eps} + \int_t^{T-\eps} (T-s)^p b_s ds \bigg| \F_t\right].$$
Hence we can pass to the limit when $\eps$ goes to zero and we obtain
$$Y^L_t \leq\frac{1}{(T-t)^p} \E \left[  \int_t^{T} (T-s)^p b_s  ds \bigg| \F_t\right]$$
which is Inequality \eqref{eq:a_priori_estimate_Y_L_2}.
\end{proof}

Next, we show that by passing to the limit $L\to \infty$ we obtain a supersolution of \eqref{eq:bsde_2} with singular terminal condition $\xi$. 
\begin{Prop} \label{thm:exists_sing_sol}
Assume that Assumptions \emph{\textbf{(A)}} hold. Let $(Y^L,\psi^L,M^L)$ be the solution of BSDE \eqref{eq:truncated_bsde} obtained in Proposition \ref{prop:exist_solution_trunc_BSDE}. Then there exists a process $(Y,\psi,M)$ such that for every $0\leq t <T$, $Y^L$ converges to $Y$ in $\bS^\ell(0,t)$, $\psi^L$ converges in $L_\pi^\ell([0,t])$ to $\psi$ and $M^L$ converges in $\cM^\ell(0,t)$ to $M$. The limit process $(Y,\psi,M)$ is a weak supersolution for the BSDE \eqref{eq:bsde_2} with singular terminal condition $\xi$.
Moreover $Y$ satisfies the estimate \eqref{eq:a_priori_estimate_Y_L}
$$Y_t \leq \frac{K_\vartheta}{(T-t)^p} \E \left[ \ \int_t^{T} \left[ \eta_s + (T-s)^p (f^0_s)^+ \right]^{\ell} ds \bigg| \F_t\right]^{1/\ell}.$$
\end{Prop}

\noindent \begin{proof}
The comparison result (see Proposition 4 in \cite{krus:popi:14}) yields that $Y^L \leq Y^N$ if $N > L$. Hence, for all $t \leq T$ we can define $Y_t$ as the increasing limit of $Y^L_t$ as $L\to \infty$. Recall that by Proposition \ref{prop:exist_solution_trunc_BSDE}, $Y^L$ is bounded from below uniformly in $L$ by some process $\bar Y \in \bS^2(0,T)$. Thus $Y$ is also bounded from below by $\bar Y$. 

By Equation \eqref{eq:a_priori_estimate_Y_L} for fixed $t < T$ the family of random variables $(Y^L_t, \ L \geq 0)$ is bounded from above:
$$Y^{L,+}_t \leq \frac{K_\vartheta}{(T-t)^p} \E \left[ \ \int_t^{T} \left[ \eta_s + (T-s)^p (f^0_s)^+ \right]^{\ell} ds \bigg| \F_t\right]^{1/\ell} .$$
Once again by Assumption \textbf{A6}, the random variable on the right hand side of in the inequality above is in $L^\ell(\Omega)$. 
By dominated convergence, $Y^L_t$ converges to $Y_t$ in $L^\ell(\Omega)$ for $t < T$.

For the convergence of $(\psi^L,M^L)$ let $0 \leq s \leq  t < T$. For $L$ and $N$ nonnegative, we put
$$\what Y_s = Y^{N}_s - Y^{L}_s, \quad \what \psi_s(z) = \psi^{N}_s(z) - \psi^{L}_s(z), \quad \what M_s = M^{N}_s - M^{L}_s.$$
Let us define $a =  \ell  \|\vartheta\|^2_{L^2_\mu}/(\ell-1)$. By Lemma \ref{lem:appendix_5} in the Appendix there exists a constant $K_\ell$ depending only on $\ell$ such that
\begin{eqnarray*}
&&\E \left[ \sup_{s\in [0,t]} e^{as} |\what Y_s|^\ell + \left(  \int_0^t e^{2au/\ell}  \int_\cZ |\what \psi_u(z)|^2\mu(dz) du \right)^{\ell/2} + \left( \int_0^t e^{2au/\ell} d[\what M]_u \right)^{\ell/2} \right]\\
&&\qquad \leq K_\ell \E \left ( e^{at} |\what Y_t|^\ell +  \int_0^te^{au} |f^0_u \wedge N- f^0_u \wedge L|^\ell du\right).
\end{eqnarray*}
Since $f^0 \in \bH^\ell(0,t)$ (see condition \textbf{A6}), the right-hand side converges to zero as $N$ and $L$ go to $+\infty$. Then $(\psi^L)$ is a Cauchy sequence in $L^\ell_\pi(0, t)$ and converges to $\psi \in L^\ell_\pi(0, t)$ for every $t < T$. The same holds for the sequence $(M^L)$ in $\cM^\ell(0,t)$. Moreover the previous inequality yields that $\ds \E\left(  \sup_{0\leq s \leq t} |Y_s|^\ell \right)  < +\infty$.

Finally, taking the limit as $L$ goes to $\infty$ in \eqref{eq:truncated_bsde} implies that $(Y,\psi,M)$ satisfies \eqref{eq:bsde_2} for every $0\leq s\leq t < T$. From the structure of the BSDE, we deduce that $Y$ is c\`adl\`ag on $[0,T)$. In other words $Y \in \bS^\ell(0,T-\eps)$ for any $\eps > 0$. 

Since the filtration is quasi-left continuous, we have: $\ds \lim_{t\nearrow T} Y^L_t = \xi \wedge L$. Indeed, in Equation \eqref{eq:truncated_bsde}, using Fubini's theorem for conditional expectation, the only discontinuous term could be the martingale term $M^L$. But the assumption on the filtration shows that $M^L$ has no jump at time $T$ (see \cite{kall:02}, Proposition 25.19). Now for any $L\geq 0$ we have
$$\liminf_{t \uparrow T} Y_t \geq \liminf_{t \uparrow T} Y^L_t = \xi \wedge L,$$
 which gives the desired inequality $\liminf_{t\nearrow T}Y_t\ge \xi$. In particular, $(\liminf_{t\nearrow T} Y_t ) \1_{\cS} = +\infty.$ This achieves the proof of the theorem.
\end{proof}

\begin{Remark}\label{stronger_upper_bound}
Under Condition \eqref{eq:f_growth_psi_2}, the estimate \eqref{eq:a_priori_estimate_Y_L_2} is then also an upper bound for $Y$. 
\end{Remark}

To finish the proof of Theorem \ref{thm:main_thm_1} let us prove the minimality of the limit process.
\begin{Prop}\label{prop:minimality}
The solution obtained in Proposition \ref{thm:exists_sing_sol} is minimal. If $(Y',\psi',M')$ is another weak supersolution of \eqref{eq:bsde_2} with terminal condition $\xi$, then $Y'_t \geq Y_t$ a.s. for all $t\in [0,T]$.
\end{Prop}
\begin{proof}
Fix $L > 0$ and let $(Y^L, \psi^L,M^L)$ denote the solution of \eqref{eq:truncated_bsde} with terminal condition $Y_T^L =\xi\wedge L$. Let $(Y',\psi',M')$ be a weak supersolution of \eqref{eq:bsde_2} in the sense of Definition \ref{def:sol_sing_BSDE}. Set
 $$\what Y_s = Y'_s - Y^{L}_s, \quad \what \psi_s(z) = \psi'_s(z) - \psi^{L}_s(z), \quad \what M_s = M'_s - M^{L}_s.$$
We have
\begin{eqnarray*}
f(t,Y'_t,\psi'_t) - f(t,Y^L_t,\psi^L_t) & = & -c_t \what Y_t  + (f(t,Y^L_t,\psi'_t) - f(t,Y^L_t,\psi^L_t))
\end{eqnarray*}
with
\begin{eqnarray*}
-c_t & =& \frac{f(t,Y'_t,\psi'_t) - f(t,Y^L_t,\psi'_t)}{\what Y_t} \1_{\what Y_t \neq 0}.
\end{eqnarray*}
Note that from condition \textbf{A1}, $-c_t \leq \chi=0$. For every $t < T$ the process $(\what Y, \what \psi, \what M)$ solves the BSDE
\begin{eqnarray*}
d\what Y_s & = & \left[ c_s \what Y_s - (f^0_s - L)^+ - (f(s,Y^L_s,\psi'_s) - f(s,Y^L_s,\psi^L_s)) \right] ds 
+ \int_\cZ \what \psi _s(z) \tpi(dz,ds) + d\what M_s
\end{eqnarray*}
on $[0, t]$ with terminal condition $\what Y_t = Y'_t - Y^L_t$. Moreover from \textbf{A2} it holds that 
$$f(s,Y^L_s,\psi'_s) - f(s,Y^L_s,\psi^L_s)\geq \int_\cZ  \kappa_s^{Y^L,\psi^L,\psi'} \what \psi_ s(z) \mu(dz).$$
From Lemma 10 in \cite{krus:popi:14} and Lemma 4.1 in \cite{quen:sule:13}, we have
\begin{eqnarray*}
\what Y_s & \geq  &  \E \left[\what Y_t \Gamma_{s,t}   + \int_s^t \Gamma_{s,u} (f^0_u - L)^+ du \bigg| \F_s \right]
\end{eqnarray*}
where $\Gamma_{s,t} = \exp\left( - \int_s^t c_u du \right) \zeta_{s,t}$ with $\zeta_{s,s}=1$ and
\begin{equation*}
d\zeta_{s,t} = \zeta_{s,t^-}  \int_\cZ \kappa_t^{Y^L,\psi^L,\psi'}\tpi(dz,dt).
\end{equation*}
Our assumptions ensure that $\zeta$ is non negative and belongs to $\bH^k(0,T)$. From Proposition \ref{prop:upper_bound_Y_L} we have $Y^L_t \leq (1+T)L$ and hence $\what Y_t \geq -((Y'_t)^-+(1+T)L)$. Thus $\what Y  \Gamma_{s,.}$ is bounded from below by a process in $\bS^m(0,T)$ for some $m>1$. We can apply Fatou's lemma to obtain
\begin{eqnarray*}
\what Y_s & = & \liminf_{t \nearrow T} \E \left[\what Y_t  \Gamma_{s,t} + \int_s^t \Gamma_{s,u}(f^0_u - L)^+ du \bigg| \F_s \right]  \geq  \E \left[  \liminf_{t \nearrow T} (\what Y_t  \Gamma_{s,t})  \bigg| \F_s \right].
\end{eqnarray*}
The process $(\Gamma_{s,t}, \ s\leq t\leq T)$ is c\`adl\`ag and non negative. Hence a.s.
$$ \liminf_{t \nearrow T} (\what Y_t  \Gamma_{s,t}) = (\liminf_{t \nearrow T} \what Y_t) \Gamma_{s,T^-} \geq (\xi - \xi\wedge L) \Gamma_{s,T^-}\geq 0.$$
Finally, $Y'_s \geq Y^L_s$ for any $s \in [0,T]$ and $L\geq 0$. Taking the limit as $L$ goes to $\infty$ yields the claim.
\end{proof}

\begin{Remark}\label{rem:Brownian_motion}
Note that all these results can be extended immediately if we assume that the filtration supports also a Brownian motion $W$ and if our singular BSDE has form
\begin{equation*}
dY_t  = f(t,Y_t,Z_t,\psi_t) dt + Z_t dW_t + \int_\cZ \psi_t(z) \tpi(dz,dt) + d M_t,
\end{equation*}
where $f$ satisfies conditions \emph{\textbf{(A)}} and is supposed to be Lipschitz continuous in $z$.
\end{Remark}

\subsection{Random terminal times}\label{bsde_rand_time}

In this section we consider the case where the terminal time $\tau$ is random. Again we proceed via truncation of the terminal condition to obtain a family of solutions $(Y^L)_{L>0}$ to \eqref{eq:truncated_bsde} with bounded terminal condition $Y^L_\tau=\xi\wedge L$.

Assumptions \textbf{A1}, \textbf{A2} and \textbf{A5} from Section \ref{bsde_det_time} remain in force, while assumptions \textbf{A2}, \textbf{A4} and \textbf{A6} are strengthened. The condition \textbf{A7} was used to construct the a priori estimate \eqref{eq:a_priori_estimate_Y_L} and is unnecessary here. Moreover, we need an extra condition between the random time $\tau$ and the growth coefficients $\chi$ in \textbf{A1} and $K$ in \textbf{A2} of $f$. This condition is denoted by \textbf{B}. Next, we present the complete list of assumptions. 
\begin{description}
\item[A1.] The function $y\mapsto f(t,y,\psi)$ is continuous and monotone: there exists $\chi \in \R$ such that a.s. and for any $t \in [0,\infty)$ and $\psi \in L^2_\mu$
\begin{equation*} 
(f(t,y,\psi)-f(t,y',\psi))(y-y') \leq \chi (y-y')^2.
\end{equation*}
\item[A2.] There exists a progressively measurable process $\kappa = \kappa^{y,\psi,\phi} : \Omega \times \R_+ \times \cZ \to \R$ such that
\begin{equation*} 
f(t,y,\psi)-f(t,y,\phi) \leq \int_\cZ (\psi(z)-\phi(z))  \kappa^{y,\psi,\phi}_t(z)  \mu(dz)
\end{equation*}
with $\P\otimes Leb \otimes \mu$-a.e. for any $(y,\psi,\phi)$, $-1 \leq \kappa^{y,\psi,\phi}_t(z)$
and $|\kappa^{y,\psi,\phi}_t(z)| \leq \vartheta(z)$ where $\vartheta \in L^2_\mu$.
As in Section \ref{bsde_det_time} we denote by $K=\|\vartheta\|_{L^2_\mu}$ is the Lipschitz constant of $f$ w.r.t. $\psi$ (cf.\ Remark \ref{rem:f_Lip}).
\end{description}
Let $\delta^*$ denote the value
\begin{equation} \label{eq:def_delta_min}
\delta^* = \left\{ \begin{array}{ll}
-\infty & \mbox{if } 2\chi < K^2, \\
K^2 + 2\chi & \mbox{if } 2|\chi| \leq K^2,\\
\chi \left( 1+ \frac{K}{\sqrt{2\chi}}\right)^2 & \mbox{if } 2\chi > K^2.
\end{array} \right.
\end{equation}
\begin{description}
\item[  B.] There exists $\rho  > \delta^*$ such that $$\E \left( e^{\rho \tau} \right)< +\infty.$$
\end{description}
If Condition \textbf{B} holds, then we put
\begin{equation}\label{eq:def_h_min}
h^*= \left\{ \begin{array}{ll}
0 & \mbox{if } 2\chi < - K^2,\\
\frac{2\rho}{\rho- \delta^*+ (\sqrt{\rho} - K\sqrt{2})^2 \mathbf{1}_{\rho > 2K^2}} & \mbox{if } 2|\chi| \leq K^2,\\
 \frac{\rho}{ \sqrt{\rho} +\sqrt{\chi} - \frac{K}{\sqrt{2}}} \times \frac{1}{\sqrt{\rho} - \sqrt{\delta^*}}& \mbox{if } 2\chi > K^2.
\end{array} \right.
\end{equation}
\begin{description}
\item[A3'.]  For every $j > 0$ and $n \geq 0$, the process $U_t(j) = \sup_{|y|\leq j} |f(t,y,0)-f^0_t|$
is in $L^1((0,n)\times \Omega)$ and there exists $\displaystyle m > h^*$ such that
$\E \int_0^\tau |U_t(j)|^m dt < +\infty$.
\item[A4'.] $\xi^-$ and $(f^0)^-$ are bounded. 
\item[A5.] There exists a constant $q > 1$ and a positive process $\eta$ such that for any $y \geq 0$
\begin{equation*} 
f(t,y,\psi)\leq -\frac{p-1}{\eta^{q-1}_t}|y|^{q} + f(t,0,\psi).
\end{equation*}
$p$ is the H\"older conjugate of $q$.
\item[A6'.] $\eta$ and $f^0$ are bounded. 
\end{description}
Note that Hypotheses \textbf{A3'} and \textbf{A5} imply that
\begin{equation} \label{eq:cond:int_eta_tau}
\E \int_0^\tau \frac{1}{\eta_s^{(q-1)m}} ds < +\infty.
\end{equation}
\begin{Remark}[on A1]
For a random terminal time, we cannot assume w.l.o.g. that $\chi=0$ in \emph{\textbf{A1}}.  
\end{Remark}
\begin{Remark}[on B and A3'] \label{rem:chi_small}
If $2\chi < -K^2$, Condition \textbf{B} is satisfied for any stopping time $\tau$ (including $\tau= +\infty$ a.s.) since one can choose $\rho< 0$ in this case. 

Note that $\delta^*$ and $h^*$ are non decreasing functions of $\chi$ and $h^*$ is a non increasing function of $\rho$, with $\lim_{\rho \to \delta^*} h^* = + \infty$ and $\lim_{\rho \to +\infty} h^* = 1$.
\end{Remark}
\noindent \textbf{Assumptions (A').}
We say that Conditions $\textbf{(A')}$ are satisfied if all following hypotheses hold:
\textbf{A1}, \textbf{A2}, \textbf{A3'}, \textbf{A4'}, \textbf{A5}, \textbf{A6'} and \textbf{B}. 
\hfill $\diamond$

\vspace{0.5 cm}
Under the above conditions, Proposition \ref{prop:exists_sol_trunc_BSDE} below shows that the truncated BSDE \eqref{eq:truncated_bsde} has a unique solution $(Y^L,\psi^L,M^L)$. The crucial difference in order to obtain a supersolution to the BSDE with {\it singular} terminal condition to the case of a deterministic terminal time, is the derivation of a uniform upper bound for the family of processes $(Y^L)$ (cf.\ Inequality \eqref{eq:a_priori_estimate_Y_L}). Example \ref{ex_explosion} below shows that in general such an upper bound does not exist and that there exist stopping times $\tau$ such that the sequence $(Y^L_t)$ converges to $\infty$ as $L\to \infty$ for $t<\tau$. Consequently one has to restrict the class of terminal times. Here we draw inspiration from \cite{popi:07}, where BSDEs with random terminal time and singular terminal condition have been studied for the first time, and consider the case where $\tau$ is given by a first exit time $\tau=\tau_D$ of a diffusion $\Xi$ from a set $D$.

More precisely, we assume that the filtration $\bF$ supports a $d$-dimensional Brownian motion $W$ which is orthogonal to $\pi$ and we introduce a forward process $\Xi$ in $\R^d$, that is a solution to the stochastic differential equation
\begin{equation}\label{eq:forwardSDE}
d\Xi_t=b(\Xi_t)dt+\sigma(\Xi_t)dW_t
\end{equation}
with some initial value $\Xi_0 \in \R^d$. The functions $b:\R^d\to \R^d$ and $\sigma:\R^d \to \R^d \times \R^d$ satisfy a global Lipschitz condition: there exists some $K>0$ such that
$$ \forall x,y\in \R^d \quad \|\sigma(x)-\sigma(y)\|+\|b(x)-b(y)\| \le K\|x-y\|.$$
Under this assumption there exists a unique strong solution $\Xi$ to \eqref{eq:forwardSDE}.
Let $D$ be an open bounded subset of $\R^d$, whose boundary is at least of class $C^2$ (see for example \cite{gilb:trud:01}, Section 6.2, for the definition of a regular boundary). From now on $\Xi_0$ is fixed and supposed to be in $D$. We define the stopping time $\tau$ as the first exit time of $D$, i.e.
\begin{equation}\label{eq:def_stop_time}
\tau = \tau_D=\inf\{t\ge 0, \quad \Xi_t \notin D\}.
\end{equation}

The condition \textbf{B} imposes some implicit hypotheses between the generator $f$, the set $D$ and the coefficients of the SDE \eqref{eq:forwardSDE}.  
In the next lemma, we give sufficient conditions to ensure \textbf{B}. Let us denote by $R$ the diameter of $D$:
$$R = \sup\{|x-y|, \ (x,y)\in D^2\},$$
by $\|\sigma\|$ the spectral norm of $\sigma$
$$\|\sigma \| = \sup_{x\in \R^d} \sup_{v \in \R^d, \ |v|\leq 1} v.(\sigma(x) \sigma^*(x)) v,$$
and by $\|b\|$ the sup norm of $b$:
$$ \|b\| =  \sup_{x\in \R^d}|b(x)|.$$
Define $j_d$ to be equal to $\pi^2/4$ if $d=1$ and to be equal to the first positive zero of the Bessel function of first kind $J_{d/2-1}$ if $d\geq 2$ (for $d=2$, $j_{2}\approx 2.4048$). 
\begin{Lemma} \label{lem:suff_cond_B}
\begin{enumerate}
\item Assume that there exists $\nu > 0$ and $v \in\R^d$ such that for all $x\in \R^d$ it holds that $b(x).v \geq \nu > 0$. If $\delta^* < \frac{\nu^2}{\|\sigma \|}$, then Condition \emph{\textbf{B}} holds for all $\rho \in (\delta^*, \frac{\nu^2}{\|\sigma \|})$. 
\item Assume that $b=0$ (there is no drift) and $\sigma \sigma^*$ is uniformly elliptic, that is there exists a constant $\al > 0$ such that $(\sigma \sigma^*)(x) \ge \alpha \Id$ for all $x\in \R^d$. If $\delta^* < \frac{2\al}{R^2} (j_{d})^2$, then Condition \emph{\textbf{B}} holds for all $\rho \in (\delta^*, \frac{2\al}{R^2} (j_{d})^2)$. 
\end{enumerate}
\end{Lemma}
\begin{proof}
Since $D$ is bounded and not equal to a singleton it holds that $0< R < +\infty$.

Assume first that there exists $\nu > 0$ and $v \in \R^d$ such that for all $x\in \R^d$, the scalar product between $b(x)$ and $v$ is bounded from below by $\nu$. W.l.o.g. we can assume that $|v|=1$. Let $t > R/\nu$. On the set $\{\tau>t\}$, it holds that $\Xi_0$ and $\Xi_s$ are in $D$. This implies on the set $\{\tau>t\}$, for any $0 \leq s \leq t$, that
$$\sup_{0\leq s \leq t} (-v).\left( \Xi_s - \Xi_0 -  \int_0^s b(\Xi_u) du \right) \geq t \nu - R .$$
Hence from Theorem II.2.2 in \cite{pins:95}
$$\P(\tau > t) \leq \P \left(\sup_{0\leq s \leq t} (-v).\left( \Xi_s - \Xi_0 -  \int_0^s b(\Xi_u) du \right) \geq t \nu - R \right) \leq \exp \left( - \frac{(t\nu -R)^2}{\|\sigma\|t}\right).$$
This implies for all $t > R/\nu$ that
$$ e^{\rho t} \P(\tau > t) \leq \exp \left( \rho t- \frac{(t\nu -R)^2}{\|\sigma\|t}\right). $$
It follows from Tonelli's theorem that
$$\E (e^{\rho \tau}) = \int_0^{+\infty} \rho e^{\rho t} \P(\tau > t) dt +1 < +\infty$$
provided that $\rho < \frac{\nu^2}{\|\sigma \|}$.

In the second case, it is known (see e.g.\ Friedman \cite{frie:76}, Theorem 14.10.1) that the condition $\E e^{\rho \tau}<\infty$ holds for all numbers $\rho$ that are smaller than the principal eigenvalue of the infinitesimal generator $\mathcal L$ of $\Xi$ on the set $D$:
$$\mathcal{L} \phi (x)= \frac{1}{2} \trace \left( \sigma(x) \sigma^*(x) D^2\phi(x) \right),$$
where $D^2 \phi$ is the Hessian matrix of $\phi \in C^2(\R^d)$. To derive a condition on $\al$ and $R$ for Assumption \textbf{B}, we consider an auxiliary problem. The set $D$ is contained in a ball $B$ of radius $R/2$ and $\tau_B$ is the first exit time of $\Xi$ from $B$. Clearly $\tau = \tau_D \leq \tau_B$. Hence we can consider the operator $\mathcal{L}$ on the ball $B$. Moreover the principal eigenvalue of $\mathcal{L}$ is greater than the one of the operator $(\al/2) \Delta$. The principal eigenvalue of the Laplace operator $\Delta$ on the unit ball is given by the constant $(j_{d})^2$. See \cite{greb:nguy:13} for details. Hence the principal eigenvalue of $(\al/2) \Delta$ on $B$ is given by $\frac{2\al}{R^2} (j_{d})^2$. Consequently, \textbf{B} holds if
$$\rho < \frac{2\al}{R^2} (j_{d})^2.$$
\end{proof}

\begin{Remark}[On A3']
The bound $\frac{\nu^2}{\|\sigma \|}$ respectively $\frac{2\al}{R^2} (j_{d})^2$ give a minimal value for the parameter $m$ in {\rm \textbf{A3'}} (see Remark \ref{rem:chi_small} and Lemma \ref{lem:optimal_delta_h} in Appendix).
\end{Remark}

Next we adapt the Definition \ref{def:sol_sing_BSDE} to the case of a random terminal time and present the main result of this section. To this end, we set 
\begin{equation}\label{eq:def_tau_eps}
\tau_\eps =\inf \{t \geq 0, \dist(\Xi_t) \leq \eps \},
\end{equation}
where $\dist(\Xi_t)$ denotes the distance between the position of $\Xi$ at time $t$ and the boundary of $D$.
\begin{Def}[Weak supersolution in the case of a random terminal time] \label{def:sol_sing_BSDE_rtt}
We say that a triple of processes $(Y,\psi,M)$ is a supersolution to the BSDE \eqref{eq:bsde_2} with singular terminal condition $Y_\tau = \xi$ if it satisfies:
\begin{enumerate}
\item $M\in \cM^\perp$, $\psi \in G_{loc}(\pi)$ and there exists some $\ell > 1$ such that for all $t\ge 0$ and $\eps>0$:
$$\E \left( \sup_{s\in [0,t]} |Y_{s \wedge \tau_\eps}|^\ell + \int_0^{t \wedge \tau_\eps} \int_\cZ |\psi_s(z)|^\ell \mu(dz) ds + [M ]^{\ell/2}_{t \wedge \tau_\eps} \right) < +\infty;$$
\item $Y$ is bounded from below by a process $\bar Y \in \bS^2(0,\tau)$;
\item for all $0\leq s \leq t$ and $\eps>0$:
\begin{eqnarray*}
Y_{s \wedge \tau_\eps}  & = & Y_{t \wedge \tau_\eps} + \int_{s \wedge \tau_\eps}^{t \wedge \tau_\eps}  f(u,Y_u,\psi_u)du - \int_{s \wedge \tau_\eps}^{t \wedge \tau_\eps} \int_\cZ \psi_u(z) \tpi(dz,du) - \int_{s \wedge \tau_\eps}^{t \wedge \tau_\eps} dM_u.
\end{eqnarray*}
\item On the set $\{t\ge \tau\}$: $Y_t=\xi, \psi=M=0$ a.s.\ and $ \ds \liminf_{t \to +\infty} Y_{t \wedge \tau} \geq \xi$ a.s.
\end{enumerate}
We say that $(Y,\psi,M)$ is a minimal supersolution to the BSDE \eqref{eq:bsde_2} if for any other supersolution $(Y',\psi',M')$ we have $Y_t\le Y'_t$ a.s.\ for any $t>0$.
\end{Def}

\begin{Thm}\label{thm:main_thm_2}
If $\tau$ is the exit time given by \eqref{eq:def_stop_time}, under Assumptions \emph{\textbf{(A')}} there exists a minimal supersolution $(Y,\psi,M)$ to \eqref{eq:bsde_2} with singular terminal condition $Y_\tau=\xi$.
\end{Thm}

As in Section \ref{bsde_det_time} we first consider the truncated BSDE \eqref{eq:truncated_bsde}.
\begin{Prop} \label{prop:exists_sol_trunc_BSDE}
Assume that Assumptions \emph{\textbf{(A')}} hold. Then there exists for each $L>0$ a solution $(Y^L,\psi^L,M^L)\in \bS^2(0,\tau) \times L^2_\pi(0,\tau)\times \cM^2(0,\tau)$ to the BSDE \eqref{eq:truncated_bsde} with terminal condition $Y^L_\tau=\xi \wedge L$.
\end{Prop}
\begin{proof}
We check that all assumptions of Theorem 3 in \cite{krus:popi:14} are satisfied. The driver $f^L$ (c.f. \eqref{eq:generator_trunc_BSDE}) of the BSDE \eqref{eq:truncated_bsde} satisfies the monotonicity condition \textbf{A1}
$$(f^L(t,y,\psi)-f^L(t,y',\psi))(y-y') \leq \chi |y-y'|^2$$
a.s. for any $(t,y,y',\psi) \in [0,T]\times \R^2\times L^2_\mu$. Moreover, from \textbf{A2}, $f^L$ is Lipschitz continuous w.r.t. $\psi$. By Condition \textbf{A3'}, $f^L$ satisfies
\begin{equation*}
\forall j > 0,\ \forall n \in \N,\quad \sup_{|y|\leq j} (|f^L(t,y,0)-f^L(t,0,0)|) \in L^1(\Omega \times (0,n)).
\end{equation*}
Moreover $|\xi\wedge L|$ and $f^L(t,0,0) = f^0_t \wedge L$ are bounded from Assumption \textbf{A4'}. The conditions \textbf{B} and \textbf{A3'} imply that there exists $r >1$ such that 
$$\delta = r \left[ \chi + \frac{K^2}{2((r-1)\wedge 1)} \right] < \rho \quad \mbox{and} \quad \frac{r\delta}{\rho - \delta} < m$$
(see Lemma \ref{lem:optimal_delta_h} in Appendix for the proof). Hence  
\begin{equation}\label{eq:cond_int_rnd}
\E \int_0^\tau e^{\delta t} (|\xi\wedge L|^r + |f^L(t,0,0)|^r) dt < +\infty.
\end{equation}

Next, let $\xi^L_t=\E[\xi\wedge L|\mathcal F_t]$ and let $(\Gamma,l,N)$ be given by the martingale representation of $\xi\wedge L$
$$\xi\wedge L=\E[\xi\wedge L]+\int_0^\infty\Gamma_sdW_s+\int_0^\infty \int_{\mathcal Z}l_s(z)\tilde \pi (dz,ds)+N_\tau.$$
Since $\xi\wedge L$ is bounded (by $L$ for $L$ large enough since $\xi^-$ is supposed to be bounded), the process $\xi_t$ is also bounded by $L$. Using Conditions \textbf{A1} and \textbf{A2} we obtain for some constant $C$ (depending on $r$) which will change from line to line:
\begin{eqnarray*}
\E\left[\int_0^\tau e^{\delta t}|f^L(t,\xi_t,l_t)|^r dt\right]&\le& C \E\left[\int_0^\tau e^{\delta t}|f(t,\xi_t,l_t)-f^0_t|^r dt\right] + C \E \int_0^\tau e^{\delta t} |f^0_t \wedge L|^{r} dt \\
& \leq &C \E\left[\int_0^\tau e^{\delta t}\|l_t\|^r_{L^2_\mu} dt\right] + C \E\left[\int_0^\tau e^{\delta t}|U_t(L)|^r dt\right]   \\
& + & C \E \int_0^\tau e^{\delta t} |f^0_t \wedge L|^{r} dt .
\end{eqnarray*}
Since $f^0$ is bounded, using \textbf{A4'} as in Inequality \eqref{eq:cond_int_rnd}, one can show that the last term is finite. By H\"older inequality, for any $h > 1$ and $\hbar > 1$ such that $(h-1)(\hbar-1)=1$
$$ \E\left[\int_0^\tau e^{\delta t}\|l_t\|^r_{L^2_\mu}  dt\right] \leq \left( \E \int_0^\tau e^{\delta h t} dt \right)^{1/h} \left( \E \int_0^\tau \|l_t\|^{r \hbar}_{L^2_\mu} dt \right)^{1/\hbar}.$$
But since $\xi \wedge L$ is bounded, the process $l$ coming from the martingale representation is in any $L^m_\pi(0,\tau)$, $m>1$. Hence choosing $h$ close enough to 1, this term is also finite. We proceed similarly for the remaining term:
$$\E\left[\int_0^\tau e^{\delta t}|U_t(L)|^r dt\right]\leq \left( \E \int_0^\tau e^{\delta h t} dt \right)^{1/h} \left( \E \int_0^\tau |U_t(L)|^{r\hbar} dt \right)^{1/\hbar}.$$
From Hypotheses \textbf{B} and \textbf{A3'} we can choose $h$ and $\hbar$ such that $\delta h < \rho$ and $r \hbar \leq m$.

Hence the assumptions of Theorem 3 in \cite{krus:popi:14} hold and there exists a solution $(Y^L,\psi^L,M^L)$ to the BSDE \eqref{eq:truncated_bsde} with terminal condition $Y_\tau=\xi \wedge L$. More precisely for any $0\leq t \leq T$
\begin{eqnarray*} \nonumber
Y^L_{t\wedge \tau} & = & Y^L_{T\wedge \tau} + \int_{t\wedge \tau}^{T\wedge \tau} \left[ f(s,Y^L_s,\psi^L_s) + (\gamma_s \wedge L) \right] ds \\
&& \qquad  - \int_{t\wedge \tau}^{T\wedge \tau}\int_\cZ \psi^L_s(z) \tpi(dz,ds) - \int_{t\wedge \tau}^{T\wedge \tau}d M^L_s,
\end{eqnarray*}
and $Y^L_t = \xi \wedge L$ on the set $\{ t \geq \tau\}$.
\end{proof}

Observe that the proof of Proposition \ref{prop:exists_sol_trunc_BSDE} does not use the fact that $\tau$ is a first hitting time but works for every stopping time $\tau$ that satisfies the integrability conditions \textbf{B} and \textbf{A3'}. Moreover if we assume
\begin{equation}\label{eq:f_growth_psi_rnd} 
|f(t,0,\psi)| \leq K^f,
\end{equation}
for some constant $K^f$, then in \textbf{B} we need simply $\rho > \chi$ (see Remark 2 in \cite{krus:popi:14}).

The next example shows that further assumptions on $\tau$ are necessary in order to ensure that the family $Y^L$ is uniformly bounded from above. Therefore we will assume the particular form \eqref{eq:def_stop_time} of $\tau$ in the sequel.

\begin{Example}\label{ex_explosion}
Assume that $\widetilde f(t,y,\psi)=-|y|^2$ and $\xi=\infty$. We assume that the filtration $\F$ supports a stopping time $\tau$ such that $\E\left[\frac 1\tau\right]=\infty$ and that satisfies the integrability conditions {\rm \textbf{B}} and \eqref{eq:cond:int_eta_tau}. This holds for example for all stopping times that have a continuous density function $f$ on $\mathbb{R}_+$ with $f(0)>0$. In particular, one can take $\tau$ to be the first jump time of a Poisson process, in which case $\tau$ is exponentially distributed. For each $L>0$ let $Y^L$ denote the solution to BSDE \eqref{eq:truncated_bsde} constructed in Proposition \ref{prop:exists_sol_trunc_BSDE}. Next, we derive a lower bound for $Y^L$. To this end let $X_t=\exp(-\int_0^tY^L_sds)$. From It\^o's formula we obtain
$$dY^L_tX_t^2=-(Y^L_tX_t)^2dt+Z^L_tX_t^2dW_t.$$
In particular, this implies $Y_0^L=\E\left[\int_0^\tau \dot X_s^2ds+LX_\tau^2\right]$. Next, fix a realization $\omega \in  \Omega$. Consider the deterministic control problem of minimizing the functional $\int_0^{\tau(\omega)}\dot x^2(s)ds+Lx^2(\tau(\omega))$ over functions $x:[0,\tau(\omega)]\to \R$ starting in $x(0)=1$ and being absolutely continuous. Using Pontryagin's maximum principle one can show that the trajectory $x(s)=\frac {\tau(\omega)-s+1/L}{\tau(\omega)+1/L}$ is optimal in this deterministic problem. In particular, it follows that
$$\int_0^{\tau(\omega)}\dot x^2(s)ds+Lx^2(\tau(\omega))=\frac 1{\tau(\omega)+1/L}\le \int_0^{\tau(\omega)} \dot X_s^2(\omega)ds+LX_{\tau(\omega)}^2(\omega)$$
Taking expectations yields $Y^L_0\ge \E\left[\frac 1{\tau+1/L}\right]$ and consequently we have by monotone convergence $\liminf_{L\to \infty}Y_0^L\ge \E\left[\frac 1\tau\right]=\infty$.
\end{Example}

The preceding example shows that we cannot expect to obtain a finite supersolution to \eqref{eq:bsde_2} with singular terminal condition and random terminal time if the terminal time occurs too suddenly. Therefore we restrict here attention to the case where $\tau$ is the first hitting time of a diffusion. We introduce the signed distance function $\dist:\R^d\to \R$ of $D$, which is defined by $\dist(x)=\inf_{y \notin D}\|x-y\|$ if $x\in D$ and $\dist(x)=-\inf_{y\in D}\|x-y\|$ if $x\notin D$. The next result is a Keller-Osserman type inequality (c.f. \eqref{uppboundrandtime} and see \cite{kell:57,osse:57}): Using analytical properties of the diffusion near the boundary $\partial D$, allows us to bound at each time $t$ the value of process $Y^L_t$ against the distance of the diffusion $\Xi$ to the boundary $\partial D$.

\begin{Prop}\label{prop:Keller-Osserman}
If $\tau$ is the exit time given by \eqref{eq:def_stop_time}, under Assumptions \emph{\textbf{(A')}} the solution processes $Y^L$ constructed in Proposition \ref{prop:exists_sol_trunc_BSDE} are bounded uniformly in $L$: There exists a process $\bar Y \in \bS^2(0,\tau)$ and a constant $C$ such that:
\begin{equation}\label{uppboundrandtime}
\bar Y_{t\wedge \tau} \leq Y_{t\wedge \tau}^L\le \frac C{ \dist(\Xi_{t\wedge \tau})^{2(p-1)}}.
\end{equation}
\end{Prop}

\begin{proof}
First observe that the lower bound of $Y^L$ follows as in Proposition \ref{prop:exist_solution_trunc_BSDE} from a comparison theorem with a BSDE with terminal condition $-\xi^-$ and driver $g(t,y,\psi)=(f(t,y,\psi)-f^0_t)-(f^0_t)^-$.

For the upper bound, let $\mu>0$ and introduce the set $D_\mu=\{x\in \R^d , \ |\dist(x)|\le \mu\}$. Then it follows from Lemma 14.16 in \cite{gilb:trud:01} that there exists a positive constant $\mu$ such that $\dist \in C^2(D_\mu)$. Since $D$ is bounded there exists a constant $R>0$ such that $0\le \dist(x)\le R$ for all $x\in \overline D$. Let $\varphi\in C^\infty(\R^d,[0,1])$ with $\varphi=1$ on $\R^d\setminus D_\mu$ and $\varphi=0$ on $D_{\mu/2}$. For $0<\epsilon\le 1$ we define a function $g\in C^2(\R^d,\R_+)$ such that $g=(1-\varphi)\dist+R\varphi+\epsilon$ on $\overline D$. Since $g\ge \epsilon$ on $\overline D$, there exists a function $\Phi \in C^2(\R^d,\R_+)$ satisfying $\Phi=Cg^{-2(p-1)}$ on $\overline D$ for any $C>0$. Observe that $\Phi$ is bounded from above by $C\dist^{-2 (p-1)}$. Next we apply It\^o's formula to the process $\Phi(\Xi_{t\wedge \tau})$. For every $t<\tau$ this yields
\begin{eqnarray*}
 && d\Phi(\Xi_t)=(p-1)\frac {\Phi^q(\Xi_t)}{\eta^{q-1}_t}dt+\nabla \Phi(\Xi_t)\sigma(\Xi_t)dW_t \\
&&\quad+\left(\nabla \Phi(\Xi_t)b(\Xi_t)+\frac 12 \trace(\sigma\sigma^*(\Xi_t)D^2\Phi(\Xi_t))-(p-1)\frac {\Phi^q(\Xi_t)}{\eta^{q-1}_t}\right)dt \\
&&=\left[ (p-1)\frac {\Phi^q(\Xi_t)}{\eta^{q-1}_t}  - f^0_t \right] dt +\nabla \Phi(\Xi_t)\sigma(\Xi_t)dW_t\\
&&+\left[f^0_t + \nabla \Phi(\Xi_t)b(\Xi_t)+\frac 12 \trace(\sigma\sigma^*(\Xi_t)D^2\Phi(\Xi_t))-(p-1)\frac {\Phi^q(\Xi_t)}{\eta^{q-1}_t}\right]dt .
\end{eqnarray*}
On $\overline D$ we have
\begin{eqnarray*}
 \Phi^r&=&C^q g^{- 2q(p-1)}= C^q g^{- 2p}\\
\nabla \Phi &=&-2(p-1)C g^{-2p+1}\nabla g \\
\frac{\partial ^2 \Phi}{\partial x_i \partial x_j}&=&-2(p-1)(-2p+1)C g^{- 2p}\frac {\partial g}{\partial x_i}\frac{\partial g}{\partial x_j}-2(p-1)C g^{-2p+1}\frac{\partial ^2 g}{\partial x_i \partial x_j}
\end{eqnarray*}
For $t\le \tau$ let
\begin{eqnarray*}
G_t&=&\nabla \Phi(\Xi_t) b(\Xi_t)+\frac 12 \trace(\sigma\sigma^*(\Xi_t)D^2\Phi(\Xi_t))-(p-1)\frac {\Phi^q(\Xi_t)}{\eta^{q-1}_t} \\
& = & -(p-1)Cg^{-2p}(\Xi_t)  H(\Xi_t)
\end{eqnarray*}
with
\begin{eqnarray*}
H(\Xi_t) & = & \frac{C^{p-1}}{\eta^{q-1}_t}+2(g\nabla g b)(\Xi_t)+(-2p+1)\|\sigma (\Xi_t)\nabla g(\Xi_t)\|^2+\left[g\trace(\sigma\sigma^*D^2g)\right](\Xi_t) \\
& \geq & \frac{C^{p-1}}{\|\eta\|^{q-1}_\infty}+2(g\nabla g b)(\Xi_t)+(-2p+1)\|\sigma (\Xi_t)\nabla g(\Xi_t)\|^2+\left[g\trace(\sigma\sigma^*D^2g)\right](\Xi_t), 
\end{eqnarray*}
since from condition \textbf{A6'}, $\eta$ is bounded. Now $\overline D$ is a compact set. Thus the continuous functions $b$ and $\sigma$ are bounded on $\overline D$. Moreover, the functions $g, \nabla g$ and $D^2g $ are bounded on $\overline D$ uniformly in $\epsilon$. Hence there exists $C_0>0$ which does not depend on $\epsilon$ such that for any $C \geq C_0$, for every $t\ge 0$ and on $\overline D$ we have $H(\Xi_{t})\ge 1$.  

Again by Assumption \textbf{A6'}, the process $f^0$ is bounded from above. Hence for some $C$ large enough:
$$-\mathcal{G}_t = G_t +f^0_t = -(p-1)Cg^{-2p}(\Xi_t)  H(\Xi_t) + f^0_t \leq -(p-1)Cg^{-2p}(\Xi_t)  + \|f^0\|_\infty \leq 0.$$

Now the constant $C$ is fixed. The process $\Phi(\Xi)$ satisfies
\begin{eqnarray*}
\Phi(\Xi_{t\wedge \tau}) & =& \Phi(\Xi_{T\wedge \tau}) + \int_{t\wedge \tau}^{T\wedge \tau}\left[ -(p-1)\frac {\Phi^q(\Xi_s)}{\eta^{q-1}_s}  +f^0_s\right] ds \\
& & \quad + \int_{t\wedge \tau}^{T\wedge \tau} \mathcal{G}_s ds-\int_{t\wedge \tau}^{T\wedge \tau}\nabla \Phi(\Xi_s)\sigma(\Xi_s)dW_s
\end{eqnarray*}
for all $0\le t\le T$, with $\mathcal{G}_s \geq 0$. Let us denote by $Z$ the martingale
$$Z_t = \int_0^t \nabla \Phi(\Xi_s)\sigma(\Xi_s)dW_s.$$
The triple $(\Phi(\Xi),0,Z)$ is solution of the BSDE with the generator:
$$v(t,y,\psi)=-(p-1)\frac {y|y|^{q-1}}{\eta^{q-1}_t}  + f^0_s + f(t,0,\psi) + \mathcal{G}_t$$
and terminal condition $\Phi(\Xi_{T\wedge \tau})=\frac C{\epsilon^{2(p-1)}}$ on $\{T\ge \tau\}$.
Condition \textbf{A5} on $f$ implies that 
$$f^L(t,\Phi(\Xi_t),0) \leq v(t,\Phi(\Xi_t),0).$$ 
Moreover we choose $\epsilon$ small enough such that $L \le C/\eps^{(p-1)/2}$. Hence $Y_{T\wedge \tau}^{L,+}\le \Phi(\Xi_{T\wedge \tau})$ on $\{T\ge \tau\}$.
The comparison principle (c.f.\ Remark 3 in \cite{krus:popi:14}) leads to: for any $t \geq 0$, $Y_{t\wedge \tau}^{L,+}\le \Phi(\Xi_{t\wedge \tau})$ and by construction $ \Phi(\Xi_{t\wedge \tau})\le C\dist^{-2 (p-1)}(\Xi_{t\wedge \tau})$. This achieves the proof.
\end{proof}

Now as in Section \ref{bsde_det_time}, we can define a process $Y$ as the limit of the increasing sequence $Y^L$ to obtain the minimal supersolution of \eqref{eq:bsde_2}. The next proposition completes the proof of Theorem \ref{thm:main_thm_2}.
\begin{Prop}
Suppose that $\tau$ is given by \eqref{eq:def_stop_time} and that Assumptions \emph{\textbf{(A')}} are in force and let $(Y^L,\psi^L,M^L)$ denote the solution of BSDE \eqref{eq:truncated_bsde} obtained in Proposition \ref{prop:exists_sol_trunc_BSDE}. Then there exists a process $(Y,\psi,M)$ such that $Y^L_t$ converges a.s. to $Y_t$, $\psi^L$ converges in $L^2_\pi(0,\tau_\epsilon)$ to $\psi$ and $M^L$ converges in $\cM^2(0,\tau_\eps)$ to $M$ for any $\eps >0$. The limit process $(Y,\psi,M)$ is the minimal supersolution for the BSDE \eqref{eq:bsde_2} with terminal condition $\xi$.
\end{Prop}

\begin{proof}
We proceed as in the proof of Proposition \ref{thm:exists_sing_sol}. We outline the main steps. First observe that $Y_t^L$ converges a.s. to a limit process $Y$ by a comparison principle (c.f.\ Remark 3 in \cite{krus:popi:14}).
Recall the definition of the stopping times $\tau_\eps$, $\eps > 0$, $\tau_\eps =\inf \{t \geq 0, \dist(\Gamma_t) \leq \eps \}.$
We have $\dist(\Gamma_{t\wedge\tau_\eps}) \geq \eps$ for $\eps$ small enough. Moreover $\tau_\eps$ converges to $\tau$ when $\eps$ goes to zero.
Using this sequence of times $\tau_\eps$, the whole sequence $(Y^L,\psi^L,M^L)$ converges to $(Y,\psi,M)$ on $\bS^2(0,\tau_\eps) \times L^2_\mu(0,\tau_\eps) \times \cM^2(0,\tau_\eps)$ for all $\eps > 0$. The main argument is that by Proposition \ref{prop:Keller-Osserman} on the interval $(0,\tau_\eps)$, the process $Y^L$ is uniformly bounded by $C/ \eps^{2(p-1)}$. Moreover $(Y,\psi,M)$ satisfies for any $\eps > 0$ and any $0 \leq t \leq T$
\begin{eqnarray*} \nonumber
Y_{t\wedge \tau_\eps} & = & Y_{T\wedge \tau_\eps} + \int_{t\wedge \tau_\eps}^{T\wedge \tau_\eps} f(s,Y_s,\psi_s) ds \\
&& \qquad  - \int_{t\wedge  \tau_\eps}^{T\wedge \tau_\eps}\int_\cZ \psi_s(z) \tpi(dz,ds) - \int_{t\wedge \tau_\eps}^{T\wedge \tau_\eps}d M_s.
\end{eqnarray*}
Since the filtration is supposed be to left-continuous, we have a.s. $\lim_{t\to +\infty} Y^L_{t\wedge \tau} = \xi \wedge L.$ Therefore we obtain the following behaviour of $Y$ at the terminal time $\liminf_{t\to +\infty} Y_{t\wedge \tau} \geq \xi.$ The minimality of the solution follows by the same arguments as in Proposition \ref{prop:minimality}.
\end{proof}

\section{Optimal Position targeting} \label{sect:back_control}

\subsection{Problem formulation}

Let us now describe the stochastic control problem. We assume that the setting from Section \ref{setting and notation} is given. Moreover, we suppose that \textbf{the measure $\mu$ is finite}. As in Section \ref{sect:exist_min_sol} we fix some $p> 1$ and denote by $q=1/(1-1/p)$ its H\"older conjugate. Let $\tau$ be a $\mathbb F$ stopping time. For any $t \in \R_+$ and $x\in \R$, we denote by $\cA(t,x)$ the set of progressively measurable processes $(X_s)_{s\ge 0}$ that satisfy the dynamics
\begin{equation}\label{eq:state_dyn}
X_s =x +\int_t^{s \vee t} \alpha_u du +\int_t^{s\vee t} \int_\cZ \beta_u(z) \pi(dz,du)
\end{equation}
for any $s  \ge 0 $ and for some $\al \in L^1(t,\infty)$ a.s. and $\beta \in G_{loc}(\pi)$. Observe that for all $X \in \cA(t,x)$ it holds that $X_s=x$ for all $s\le t$. 
We consider the stochastic control problem to minimize the functional\footnote{We use the convention that $0\cdot \infty := 0$}
\begin{equation}\label{eq:control_pb}
J(t,X) = \E \left[  \int_{t\wedge \tau}^\tau \left( \eta_s |\alpha_s|^p + \gamma_s |X_s|^p + \int_\cZ \lambda_s(z) |\beta_s(z)|^p \mu(dz) \right) ds  + \xi |X_\tau|^p \bigg| \F_t \right]
\end{equation}
over all $X\in \cA(t,x)$. The random variable $\xi$ is supposed to be non negative and may take the value $\infty$ with positive probability. Observe that if for $x>0$ there exists $X \in \cA(t,x)$ such that $J(t,X)<\infty$, then $\tau>t$ a.s. and $X$ satisfies almost surely that 
\begin{equation} \label{eq:terminal_cons}
X_\tau  \1_{\xi=\infty}= 0.
\end{equation}
This way we impose implicitly a terminal state constraint on the set of admissible controls.
For future reference we define the set $\cS$ by $\cS = \{ \xi = +\infty\}.$
The coefficient processes $(\eta_t)_{t\ge 0}$, $(\gamma_t)_{t\ge 0}$ and $(\lambda_t)_{t\ge 0}$ are nonnegative progressively measurable c\`adl\`ag processes. The process $\lambda$ is $\widetilde{\mathcal{P}}$-measurable with values in $[0,+\infty]$.

We introduce the random field $v$ that represents for each initial condition $(t,x)$ the minimal value of $J$
\begin{equation}\label{eq:value_fct}
v(t,x) = \essinf_{X\in \cA(t,x)} J(t,X).
\end{equation}

Theorem \ref{thm:main_thm_3} below summarizes the main results of this section. It shows that the value function $v$ and optimal controls of the control problem \eqref{eq:value_fct} are characterized by the BSDE \eqref{eq:bsde_contr_prob} with singular terminal condition
\begin{equation} \tag{\ref{eq:bsde_contr_prob}}
dY_t  = (p-1)  \frac{Y_t^{q}}{\eta_t^{q-1}} dt + \Theta(t,Y_t,\psi_t) dt - \gamma_t dt + \int_\cZ \psi_t(z) \tpi(dz,dt) + d M_t
\end{equation}
where the function $\Theta$ is given by
\begin{equation}\label{eq:generator}
\Theta(t,y,\psi) = \int_\cZ (y+  \psi(z)) \left( 1- \frac{\lambda_t(z)}{\left((y+  \psi(z))^{q-1}+ \lambda_t(z)^{q-1} \right)^{p-1}}  \right) \1_{y+  \psi(z) \geq 0} \ \mu(dz).
\end{equation}

Again we distinguish two cases. In the first case we assume that $\tau$ is deterministic and impose some integrability assumptions on the coefficient processes $(\eta_t)_{t\ge 0}$ and $(\gamma_t)_{t\ge 0}$.

\noindent \textbf{Assumption (C1).}
The stopping time $\tau$ is a.s.\ equal to a deterministic constant $T>0$. The process $\eta$ is positive, the process $\gamma$ is non negative, such that for some $\ell > 1$
$$ \E\left[\int_0^T (\eta_t + (T-t)^p\gamma_t )^\ell dt\right]<\infty \quad \text{ and } \quad \E\left[\int_0^T\frac 1{\eta_t^{q-1}}dt\right]<\infty.$$
\hfill $\diamond$

In the second case we assume that $\tau$ is given by \eqref{eq:def_stop_time} as the first hitting time of a diffusion. We need to impose some stronger boundedness conditions on $\eta$ and $\gamma$ compared to \textbf{(C1)}.

\noindent \textbf{Assumption (C2).}
We have $\tau=\tau_D$ and there exists $\rho >\mu(\cZ)$ such that $\E e^{\rho \tau}<\infty$. The processes $\eta$ and $\gamma$ are bounded from above, $\eta$ is positive and satisfies the integrability conditions
 \begin{equation}\label{low_bound_eta_stopp_time}
\E\left[\int _0^n\frac 1{\eta_t^{q-1}}dt\right] + \E\left[ \int_0^\tau \frac 1{\eta_t^{m(q-1)}} dt \right] <\infty
\end{equation}
for all $n\in \N$ and for some $m$ satisfying:
$$m > \frac{2\rho}{\rho - \mu(\cZ) + (\sqrt{\rho} - \sqrt{2\mu(\cZ)})\mathbf{1}_{\rho > 2\mu(\cZ)}}.$$ 
The process $\gamma$ is non negative.
\hfill $\diamond$

Lemma \ref{lem:suff_cond_B} gives sufficient conditions on the coefficients of the forward SDE \eqref{eq:forwardSDE} such that $\E e^{\rho \tau}<\infty$ holds. 

\begin{Thm} \label{thm:main_thm_3}
Let Assumptions \emph{\textbf{(C1)}} or \emph{\textbf{(C2)}} hold. Then there exists a minimal supersolution $(Y,\psi,M)$ to \eqref{eq:bsde_contr_prob} with singular terminal condition $Y_\tau=\xi$. Set $Y_s=\xi$ for all $s\ge \tau$. For all $(t,x)\in \R_+\times \R$ it holds $\P$-a.s.\ that $v(t,x)=Y_tx^p$. Moreover, for every $(t,x)\in \R_+\times \R$ the process $X$ satisfying the linear dynamics
$$X_s=x-\int_t^{s\vee t}\left(\frac{Y_u}{\eta_u}\right)^{q-1}X_udu-\int_t^{s\vee t}X_{u-}\int_\cZ\zeta_u(z)\pi(dz,du),$$
with
$$\zeta_u(z)=\frac{(Y_{u^-} + \psi_u(z))^{q-1}}{ \left[ (Y_{u^-} + \psi_u(z))^{q-1} + \lambda_u(z)^{q-1}\right]}$$
belongs to $\cA(t,x)$, satisfies the terminal state constraint \eqref{eq:terminal_cons} if $t<\tau$ and is optimal in \eqref{eq:value_fct}.
\end{Thm}
The optimal process $X^*$ is given explicitely by
\begin{equation}\label{eq:optim_contr_sing}
X^*_s = x\exp \left[ - \int_t^{s\vee t} \left(  \frac{Y_u}{\eta_u} \right)^{q-1} du \right] \exp \left[ \int_t^{s\vee t} \int_\cZ  \ln \left( 1 - \zeta_u(z)\right) \pi(dz,du) \right].
\end{equation}

To prove Theorem \ref{thm:main_thm_3} we first conclude from Theorems \ref{thm:main_thm_1} or \ref{thm:main_thm_2} that there exists a minimal supersolution to \eqref{eq:bsde_contr_prob}. We then consider a variant of the minimization problem \eqref{eq:value_fct}, where we penalize any non zero terminal state by $(\xi\wedge L)|X_\tau|^p$ and thus omit the constraint $X_\tau\1_\cS=0$ on the set of admissible controls. We show that optimal controls for this unconstrained minimization problem admit a representation in terms of the solutions $Y^L$ of a truncated version of \eqref{eq:bsde_contr_prob}. We then use this result to derive an optimal control for \eqref{eq:value_fct}.

\subsection{Existence of a minimal supersolution}

Observe that BSDE \eqref{eq:bsde_contr_prob} is a special case of \eqref{eq:bsde_2} with generator $f$ given by
$$f(t,y,\psi) = -(p-1)  \frac{y|y|^{q-1}}{\eta_t^{q-1}} - \Theta(t,y,\psi) + \gamma_t.$$
Recall that in this section $\mu$ is supposed to be a finite measure, thus $\Theta$ (given by \eqref{eq:generator}) is well-defined. Here we have that $f^0_t = f(t,0,0) = \gamma_t$. 
For simplicity we denote by $\varpi$ the function
$$\varpi(t,y,\phi) = (y+  \phi) \left( 1- \frac{\lambda_t(z)}{\left((y+  \phi)^{q-1}+ \lambda_t(z)^{q-1} \right)^{p-1}}  \right) \1_{y+  \phi \geq 0}$$
such that
$$\Theta(t,y,\psi) = \int_\cZ \varpi(t,y,\psi(z)) \mu(dz).$$

The next result is a consequence of Theorems \ref{thm:main_thm_1} and \ref{thm:main_thm_2}.
\begin{Coro}\label{coro:exist_sing_control}
Under Assumptions \emph{\textbf{(C1)}} or \emph{\textbf{(C2)}}, the singular BSDE \eqref{eq:bsde_contr_prob} has a minimal non negative weak supersolution $(Y,\psi,M)$.
\end{Coro}
\begin{proof}
We have to prove that $f$ satisfies Conditions \textbf{(A)} (respectively \textbf{(A')}) if \textbf{(C1)} (respectively \textbf{(C2)}) holds.
A simple computation proves that for a fixed $(t,\psi) \in [0,T] \times L^2_\mu$ and $z \in \cZ$, the function $y \mapsto \varpi(t,y,\psi(z))$ is non decreasing and of class $C^1$ on $\R$ with a derivative bounded by 1
$$\frac{\partial \varpi}{\partial y}(t,y,\psi(z)) = \left( 1 - \frac{\lambda_t(z)^q}{\left((y+  \psi(z))^{q-1}+ \lambda_t(z)^{q-1} \right)^{p}} \right)\1_{y+\psi(z)\geq 0}.$$
Since $\eta > 0$, the condition \textbf{A1} is satisfied with $\chi=0$.

From the same argument the function $\varpi$ is Lipschitz continuous w.r.t. $\psi(z)$ and hence we obtain
$$|\Theta(t,y,\psi) - \Theta(t,y,\psi')| \leq \int_\cZ |\psi(z) - \psi'(z)| \mu(dz) \leq \mu(\cZ)^{1/2} \|\psi-\psi'\|_{L^2_\mu}.$$
Moreover for any $(t,y,\psi,\psi') \in [0,T]\times \R\times (L^2_\mu)^2$ we have
\begin{eqnarray*}
f(t,y,\psi)-f(t,y,\psi') & = & -\Theta(t,y,\psi)+\Theta(t,y,\psi') = \int_\cZ (\varpi(t,y,\psi'(z)) - \varpi(t,y,\psi(z)))\mu(dz) \\
& = & \int_\cZ (\psi(z)-\psi'(z))  \kappa_t^{y,\psi,\psi'}(z)  \mu(dz)
\end{eqnarray*}
where
$$ \kappa_t^{y,\psi,\psi'}(z) =- \frac{\varpi(t,y,\psi(z)) - \varpi(t,y,\psi'(z))}{\psi(z)-\psi'(z)}\1_{\psi(z)\neq \psi'(z)}.$$
Since $\varpi$ is non decreasing in $\psi$ with derivative bounded from above by $1$, we obtain $-1\leq  \kappa_t^{y,\psi,\psi'} \leq 0.$ Thus Conditions \textbf{A2} and \textbf{A7} hold for any $k \geq 1$. We can even note that \eqref{eq:f_growth_psi_2} (cf. Lemma \ref{lem:another_estimate} and Remark \ref{stronger_upper_bound}) is true with $K^f_t=0$. For every $r>0$ and $|y|\leq r$ we have
$$|f(t,y,0) - f^0_t| = (p-1)\frac{|y|^{q}}{\eta_t^{q-1}} + |\Theta(t,y,0)|  \leq  (p-1)\frac{|r|^{q}}{\eta_t^{q-1}} + \mu(\cZ)|r| =: U_t(r).$$
By Assumption \textbf{(C1)}, the mapping $t\mapsto U_t(r)$ is in $L^1((0,T)\times \Omega)$ and Condition \textbf{A3} holds. Condition \textbf{A4} holds since $\gamma$ and $\xi$ are non negative. Finally since $\Theta \geq 0$, Condition \textbf{A5} is satisfied and \textbf{A6} holds if Assumption \textbf{(C1)} is assumed.

A similar computation shows that under \textbf{(C2)}, Conditions \textbf{A4'} and \textbf{A6'} hold. We have here $\chi=0$ and $K^2 = \mu(\cZ)$, thus $\delta^* = \mu(\cZ)$ (see Equation \eqref{eq:def_delta_min}) and therefore the assumption $\rho > \mu(\cZ)$ implies Condition \textbf{B}. Moreover from \eqref{low_bound_eta_stopp_time}, the process $U_t(r)$ is in $L^1((0,n)\times \Omega)$ for any $n\in \N$ and satisfies
$\E \int_0^\tau |U_t(r)|^m dt <+\infty$, with $m > h^*$ (see Equation \eqref{eq:def_h_min}). Hence Corollary \ref{coro:exist_sing_control} is a direct consequence of Theorems \ref{thm:main_thm_1} or \ref{thm:main_thm_2}. Moreover, by Proposition \ref{prop:exist_solution_trunc_BSDE} (respectively Proposition \ref{prop:exists_sol_trunc_BSDE}) there exists a solution $(Y^L,\psi^L,M^L)$ of the truncated BSDE
\begin{equation}\label{eq:control_trunc_bsde}
dY^L_t = (p-1)\frac{(Y^L_t)^{1+q}}{\eta_t^{q}} dt + \Theta(t,Y^L_t,\psi^L_t) dt - (\gamma_t\wedge L) dt + \int_\cZ \psi^L_t(z) \tpi(dz,dt) +dM^L_t
\end{equation}
with terminal condition $Y^L_\tau = \xi \wedge L$. The process $(Y,\psi,M)$ is the limit as $L$ goes to $+\infty$ of $(Y^L,\psi^L,M^L)$ and is the minimal (super-)solution of the BSDE \eqref{eq:bsde_contr_prob}.
\end{proof}

\subsection{Penalization}

For $L>0$ and $(t,x)\in \R_+\times \R$ we consider the unconstrained minimization problem:
\begin{eqnarray} \nonumber
v^L(t,x) &=& \essinf_{X\in \cA(t,x)} J^L(t,X) \\ \nonumber
&  = & \essinf_{X\in\cA(t,x)}  \E \left[  \int_{t\wedge \tau}^\tau \left( \eta_s |\alpha_s|^p + (\gamma_s \wedge L) |X_s|^p + \int_\cZ \lambda_s(z) |\beta_s(z)|^p \mu(dz) \right) ds  \right. \\  \label{eq:unconstraint_value_fct}
&& \qquad \qquad  \qquad  \qquad \qquad \left. + (\xi \wedge L) |X_\tau|^p \bigg| \F_t \right].
\end{eqnarray}
\begin{Prop} \label{prop:optimal_control_uncons}
Let Assumption \emph{\textbf{(C1)}} or \emph{\textbf{(C2)}} hold and let $(Y^L, \psi^L,M^L)$ be the solution to \eqref{eq:control_trunc_bsde} with terminal condition $Y_\tau=\xi\wedge L$. Let $Y_s=L\wedge \xi$ for all $s\ge \tau$. Then for all $(t,x)\in \R_+\times \R$ the process $X^L$ satisfying the linear dynamics
$$X^L_s=x-\int_t^{s\vee t}\left(\frac{Y^L_r}{\eta_r}\right)^{q-1}X^L_rdr-\int_t^{s\vee t}X^L_{r-}\int_\cZ\zeta^L_r(z)\pi(dz,dr),$$
with
$$\zeta^L_r(z)=\frac{(Y^L_{r^-} + \psi_r(z))^{q-1}}{ \left[ (Y^L_{r^-} + \psi^L_r(z))^{q-1} + \lambda_r(z)^{q-1}\right]}$$
is optimal in \eqref{eq:unconstraint_value_fct}. Moreover, we have $v^L(t,x) = Y^L_t |x|^p$.
\end{Prop}

To prove Proposition \ref{prop:optimal_control_uncons} we will make use of the two following auxiliary results. The first lemma shows that in the case $x\ge 0$ we can without loss of generality restrict attention to monotone strategies\footnote{It is straightforward to show that $v(t,x)=v(t,-x)$ for all $(t,x)\in \R_+\times \R_+$. Therefore, we restrict attention to the case $x\ge 0$ in the sequel.}. To this end we introduce the set $\cD(t,x)$, the subset of $\cA(t,x)$ containing only processes $X$ that have nonincreasing sample paths (i.e.\ $\alpha_t \leq 0$ and $\beta_t(z) \leq 0$), and that remain nonnegative.

\begin{Lemma} \label{lem:decreasing_control}
Let $x\ge 0$. Every control $X\in \cA(t,x)$ can be modified to a control $\underline{X} \in \cD(t,x)$ such that $J^L(t,X) \geq J^L(t,\underline{X})$. In particular, $v^L(t,x) = \essinf_{X\in \cD(t,x)} J^L(t,X) $.
\end{Lemma}
\begin{proof}
For $s\ge 0$ we consider the solution of the following SDE
$$\wtil X_s = x - \int_t^{s\vee t} \al_u^- du - \int_t^{s\vee t} \int_{\cZ} \beta_s(z)^- \pi(dz,ds),$$
where $x^-$ denotes the negative part of $x$. This process is nonincreasing and satisfies $\wtil X_s \leq X_s$. Then we define
$$\underline{X}_s = \wtil X_s \vee 0 = (\wtil X_s)^+.$$
By Tanaka's formula we have
$$\underline{X}_s  = x - \int_t^{s\vee t} \1_{\wtil X_u > 0} \al_u^- du - \int_t^{s\vee t} \int_{\cZ}  \1_{\wtil X_{u^-} > 0}  (\beta_u(z)^- \wedge (\wtil X_{u^-})^+) \pi(dz,ds).$$
We define
$$\what \al_s = -\1_{\wtil X_s > 0} \al_s^-, \qquad \what \beta_s(z) = - \1_{\wtil X_{s^-} > 0}  (\beta_s(z)^- \wedge (\wtil X_{s^-})^+).$$
Then $\underline{X}$ belongs to $\cD(t,x)$. Moreover we have
$$|\what \al_s| \leq |\al_s|, \quad |\what \beta_s(z)| \leq |\beta_s(z)|, \quad 0\le \underline{X}_s \leq |X_s|$$
which implies that $J^L(t,X) \geq J^L(t,\underline{X})$.
\end{proof}

The second lemma provides the dynamics of two auxiliary processes.

\begin{Lemma}\label{Lemma:mart_property}
 Let Assumptions \emph{\textbf{(C1)}} or \emph{\textbf{(C2)}} hold and let $(Y^L,\psi^L, M^L)$ be the solution of \eqref{eq:control_trunc_bsde}. Let $X^L \in \cA(t,x)$ be the strategy from Proposition \ref{prop:optimal_control_uncons}. Then we have for $t\le s\le \tau$ that
 $$d\left(\eta_s |\alpha^L_s|^{p-1}\right)= (X^L_{s^-})^{p-1}dM^L_s- (\gamma_s\wedge L) |X^L_s|^{p-1} ds- \int_\cZ \phi_s(z)\tpi(dz,ds),$$
with $\phi_s(z)=Y^L_s |X^L_{s^-}|^{p-1} -  \lambda_s(z) |\beta^L_s(z)|^{p-1}$. Moreover, we have for $t\le s\le \tau$
\begin{eqnarray*}
d(Y^L_s (X^L_s)^p)& = & -\left[ \eta_s |\al^L_s|^p + \gamma_s^L (X^L_{s})^{p}  + \int_\cZ \lambda_s(z) |\beta^L_s(z)|^p \mu(dz) \right] ds\\
&&\quad + (X^L_{s^-})^{p} dM^L_s + (X^L_{s^-})^{p}   \int_\cZ (Y^L_{s^-} + \psi^L_s(z)) \left[ \left( 1-\zeta^L_s(z) \right)^{p} - 1 \right] \tpi(dz,ds)
\end{eqnarray*}
\end{Lemma}
\begin{proof}
 To simplify notation we set $\gamma^L_s = \gamma_s \wedge L$. Recall that $X^L$ and $Y^L$ satisfy the following dynamics for $t\le s\le \tau$
\begin{eqnarray*}
dX^L_s & = &  -  \frac{(Y^L_s)^{q-1}}{\eta_s^{q-1}} X^L_s ds -  \int_\cZ  X^L_{s^-}  \zeta^L _s(z)  \pi(dz,ds),\\
dY^L_s & = & \left[  (p-1)  \frac{(Y^L_s)^{q}}{\eta_s^{q-1}} +  \vth (s,Y^L_s,\psi^L_s)  - \gamma^L_s  \right] ds + \int_\cZ \psi^L_s(z) \tpi(dz,ds) +dM^L_s
\end{eqnarray*}
For $t\le s\le \tau$ let
$$\theta_s = \eta_s |\alpha^L_s|^{p-1} + \int_t^s \gamma^L_u |X^L_u|^{p-1} du= Y^L_s |X^L_s|^{p-1} +  \int_t^s \gamma^L_u |X^L_u|^{p-1} du.$$
Applying the integration by parts formula to $\theta$ results in
\begin{eqnarray*}
d\theta_s & = & (X^L_{s^-})^{p-1} dY^L_s + Y^L_{s^-} d( (X^L_{s})^{p-1}) + d [ Y^L,(X^L)^{p-1} ]_s + \gamma^L_s |X^L_s|^{p-1}ds \\
 &= &  (X^L_{s^-})^{p-1} dY^L_s + Y^L_{s^-}(X^L_{s^-})^{p-1}  \left( -(p-1)  \frac{(Y^L_s)^{q-1}}{\eta_s^{q-1}} \right) ds \\
&& \quad + Y^L_{s^-}(X^L_{s^-})^{p-1} \int_\cZ \left( \left( 1-\zeta^L_s(z)  \right)^{p-1} - 1 \right) \mu(dz) ds \\
&& \quad  + Y^L_{s^-}(X^L_{s^-})^{p-1}   \int_\cZ \left( \left( 1-\zeta^L_s(z) \right)^{p-1} - 1 \right) \tpi(dz,ds) \\
&& \quad + (X^L_{s^-})^{p-1}   \int_\cZ \psi^L_s(z) \left( \left( 1-\zeta^L_s(z) \right)^{p-1} - 1 \right) \pi(dz,ds) + p\gamma^L_s |X^L_s|^{p-1}ds \\
& = & (X^L_{t^-})^{p-1}  \Theta (s,Y^L_s,\psi^L_s) ds + (X^L_{s^-})^{p-1} \int_\cZ (Y^L_{s^-} + \psi^L_s(z)) \left( \left( 1-\zeta^L_s(z)  \right)^{p-1} - 1 \right) \mu(dz) ds \\
&& \quad (X^L_{s^-})^{p-1}dM^L_s + (X^L_{s^-})^{p-1} \int_\cZ (Y^L_{s^-} + \psi^L_s(z)) \left( \left( 1-\zeta^L_s(z)  \right)^{p-1} - 1 \right) \tpi(dz,ds) \\
& = & (X^L_{s^-})^{p-1}dM^L_s + (X^L_{s^-})^{p-1} \int_\cZ (Y^L_{s^-} + \psi^L_s(z)) \left( \left( 1-\zeta^L_s(z)  \right)^{p-1} - 1 \right) \tpi(dz,ds)
\end{eqnarray*}
from the definition of $\zeta^L$ and $\Theta$ (see Equation \eqref{eq:generator}). Moreover we have
$$(Y^L_{s^-} + \psi^L_s(z))  \left[ \left( 1-\zeta^L_s(z) \right)^{p-1} - 1 \right] =     \lambda_s(z) \zeta^L_s(z)^{p-1} -(Y^L_{s^-}+\psi^L_s(z)),$$
which yields the first claim.

For the second equation we apply the integration by parts
formula to the process $Y^L (X^L)^p$ to obtain
\begin{eqnarray*}
d(Y^L_s (X^L_s)^p)& = & (X^L_{s^-})^{p} dY^L_s + Y^L_{s^-} d( (X^L_{s})^{p}) + d [ Y^L,(X^L)^{p} ]_s  \\
& = & -\left[ \eta_s (X^L_{s})^{p} \frac{(Y^L_s)^{q}}{\eta_s^{q}} + \gamma_s^L (X^L_{s})^{p} \right] ds + (X^L_{s^-})^{p} dM^L_s\\
&& \quad +(X^L_{s^-})^{p}\Theta(s,Y^L_s,\psi^L_s) ds \\
&& \quad + (X^L_{s^-})^{p} \int_\cZ (Y^L_{s^-} + \psi^L_s(z)) \left[ \left( 1-\zeta^L_s(z)  \right)^{p} - 1 \right] \mu(dz) ds\\
&& \quad + (X^L_{s^-})^{p}   \int_\cZ (Y^L_{s^-} + \psi^L_s(z)) \left[ \left( 1-\zeta^L_s(z) \right)^{p} - 1 \right] \tpi(dz,ds).
\end{eqnarray*}
But note that
$$|\al^L_s|^p =  \left|  \frac{(Y^L_s)^{q-1}}{\eta_s^{q-1}} X^L_s \right|^p =\frac{(Y^L_s)^{q}}{\eta_s^{q}} (X^L_s)^p,$$
and from the very definition \eqref{eq:generator} of $\Theta$
\begin{eqnarray*}
&& \Theta(s,Y^L_s,\psi^L_s) + \int_\cZ (Y^L_{s} + \psi^L_s(z)) \left[ \left( 1-\zeta^L_s(z) \right)^{p} - 1 \right] \mu(dz) \\
&& = \int_\cZ (Y^L_{s} + \psi^L_s(z)) \left[ \left(  \frac{\lambda_s(z)^{q-1}}{ \left[ (Y^L_{s^-} + \psi^L_s(z))^{q-1} + \lambda_s(z)^{q-1}\right]} \right)^{p} \right. \\
&& \qquad  \qquad  \qquad  \qquad  \qquad \qquad  \left.- \frac{\lambda_s(z)}{\left(|Y^L_{s} +  \psi^L_s(z)|^{q-1}+ \lambda_s(z)^{q-1} \right)^{p-1}} \right] \mu(dz) \\
&& =- \int_\cZ (Y^L_{s} + \psi^L_s(z)) \frac{\lambda_s(z)}{ \left[ (Y^L_{s^-} + \psi^L_s(z))^{q-1} + \lambda_s(z)^{q-1}\right]^p} \left[ \left(Y^L_{s} +  \psi^L_s(z)\right)^{q-1}  \right] \mu(dz)\\
&&=- \int_\cZ\lambda_s(z)|\zeta_s(z)|^p \mu(dz).
\end{eqnarray*}
\end{proof}

We close this section with the proof of Proposition \ref{prop:optimal_control_uncons}.

\noindent \begin{proof}[Proof of Proposition \ref{prop:optimal_control_uncons}]
We omit the superscript $L$ in the sequel. Let $(t,x)\in \R_+\times \R_+$. Take another process $\bX$ in $\cD(t,x)$. Use the convexity of the function $y \mapsto |y|^p$ and $\alpha_s \leq 0$ to obtain
\begin{eqnarray}\label{eq:mart_prop_1}\nonumber
 && \int_{t\wedge \tau}^\tau \left( \eta_s (|\alpha_s|^p -|\ba_s|^p) \right) ds\le -p\int_{t\wedge \tau}^\tau \eta_s|\alpha_s|^{p-1} \left(\alpha_s - \ba_s \right) ds \\ \nonumber
&& \quad = -p\int_{t\wedge \tau}^\tau \eta_s|\alpha_s|^{p-1}(dX_s  - d\bX_s)+p\int_{t\wedge \tau}^\tau \int_\cZ \eta_s|\alpha_s|^{p-1} \left(\beta_s(z) - \bb_s(z) \right) \pi(dz,ds) \\
&& \quad = \mathcal{I}^1_t + \mathcal{I}^2_t
\end{eqnarray}
By integration by parts on the first integral and using Lemma \ref{Lemma:mart_property} and boundedness of $X$ and $\bX$ (see Lemma \ref{lem:decreasing_control}), we obtain
\begin{eqnarray*}
 \E^{\F_t}\mathcal{I}^1_t  & = & -p\E^{\F_t}\left[\eta_\tau|\alpha_\tau|^{p-1}(X_\tau  -\bX_\tau)\right] +p \E^{\F_t}\left[\int_{t\wedge \tau}^\tau(X_s  - \bX_s)d\left(\eta_s|\alpha_s|^{p-1}\right)\right]\\
&&\quad -p \E^{\F_t}\left[\int_{t\wedge \tau}^\tau\int_\cZ  \left(\beta_s(z) - \bb_s(z) \right) \phi_s(z) \pi(dz,ds)\right] \\
&=&-p\E^{\F_t}\left[Y_\tau X_\tau ^{p-1}(X_\tau   -\bX_\tau )\right]-p\E^{\F_t}\left[\int_{t\wedge \tau}^\tau (\gamma_s\wedge L) |X^L_s|^{p-1}(X_s -\bX_s) ds\right] \\
&& \quad -p \E^{\F_t}\left[\int_{t\wedge \tau}^\tau \int_\cZ  \left(\beta_s(z) - \bb_s(z) \right) \phi_s(z) \mu(dz)ds\right]
\end{eqnarray*}
where $\phi$ is defined as in Lemma \ref{Lemma:mart_property}. Using again convexity of $y \mapsto |y|^p$ yields
\begin{eqnarray}\label{eq:mart_prop_2}\nonumber
\E^{\F_t}\mathcal{I}^1_t&\le& -\E^{\F_t}\left[(\xi\wedge L)(X_\tau ^p  -\bX_\tau ^p))\right]-\E^{\F_t}\left[\int_{t\wedge \tau}^\tau (\gamma_s\wedge L) (X_s^p  -\bX_s^p) ds\right]\\
&& \quad -p \E^{\F_t}\left[\int_{t\wedge \tau}^\tau \int_\cZ  \left(\beta_s(z) - \bb_s(z) \right) \phi_s(z) \mu(dz)ds\right].
\end{eqnarray}
Moreover we have
\begin{eqnarray}\label{eq:mart_prop_3}
 \E^{\F_t} \mathcal{I}^2_t  &= &  p\E^{\F_t} \int_{t\wedge \tau}^\tau  \int_\cZ \eta_s|\alpha_s|^{p-1}\left(\beta_s(z) - \bb_s(z) \right) \mu( dz)ds
\end{eqnarray}
Now, using \eqref{eq:mart_prop_1}, \eqref{eq:mart_prop_2} and \eqref{eq:mart_prop_3} we obtain
\begin{eqnarray*}
J(t,X) - J(t,\bX) &\le& \E^{\F_t} \left[\int_{t\wedge \tau}^\tau   \int_\cZ p\left(\beta_s(z) - \bb_s(z) \right)\left(\phi_s(z)-\eta_s|\alpha_s|^{p-1}\right)\mu(dz) ds \right]\\
&&\quad+\E^{\F_t} \left[\int_{t\wedge \tau}^\tau   \int_\cZ\lambda_s(z) \left( |\beta_s(z)|^p- |\bb_s(z)|^p \right) \mu(dz)ds\right].
\end{eqnarray*}
Now recall that $\eta_s|\alpha_s|^{p-1} = Y^L_s |X^L_{s}|^{p-1}$. From the definition of $\phi_s$ and from convexity of $x \mapsto |x|^p$ we obtain:
\begin{eqnarray*}
J(t,X) - J(t,\bX) &\le& \E^{\F_t} \left[\int_{t\wedge \tau}^\tau   \int_\cZ pY^L_s\left(\beta_s(z) - \bb_s(z) \right)\left( |X^L_{s^-}|^{p-1} - |X^L_{s}|^{p-1} \right)\mu(dz) ds \right]
\end{eqnarray*}
and therefore $J(t,X) - J(t,\bX)\le 0$.

It remains to verify the identity $v^L(t,x) = Y^L_t |x|^p$.
But from Lemma \ref{Lemma:mart_property} we deduce that
\begin{eqnarray*}
Y^L_t |x|^p &=& \E^{\F_t} \int_{t\wedge \tau}^\tau  \left[ \eta_u |\al^L_u|^p + \gamma_u^L (X^L_{u})^{p} + \int_\cZ \lambda_u(z) |\beta^L_u(z)|^p \mu(dz) \right] du + \E^{\F_t} (Y^L_\tau  |X^L_\tau |^p)\\
&=&J(t,X)=v^L(t,x).
\end{eqnarray*}
\end{proof}

\subsection{Solving the constrained problem}
This section is devoted to the proof of Theorem \ref{thm:main_thm_3}. For the convenience of the reader we restate the result here.
\begin{Thm}\label{thm:verif_result}
Let Assumptions \emph{\textbf{(C1)}} or \emph{\textbf{(C2)}} hold and let $(Y,\psi,M)$ be the minimal solution to \eqref{eq:bsde_contr_prob} with singular terminal condition $Y_\tau =\xi$ from Corollary \ref{coro:exist_sing_control} and let $Y_s=\xi$ for all $s\ge \tau$. Then $v(t,x) = Y_t |x|^p $ for all $(t,x)\in \R_+\times \R$. Moreover the control given by Equation \eqref{eq:optim_contr_sing}
\begin{equation*}
X^*_s = x\exp \left[ - \int_t^{s\vee t} \left(  \frac{Y_u}{\eta_u} \right)^{q-1} du \right] \exp \left[ \int_t^{s\vee t} \int_\cZ  \ln \left( 1 - \zeta_u(z)\right) \pi(dz,du) \right]
\end{equation*}
with
$$ \zeta_t(z) = \frac{(Y_{t^-} + \psi_t(z))^{q-1}}{ \left[ (Y_{t^-} + \psi_t(z))^{q-1} + \lambda_t(z)^{q-1}\right]}$$
belongs to $\cA(t,x)$, satisfies the terminal state contraint \eqref{eq:terminal_cons} if $t<\tau$ and is optimal in \eqref{eq:value_fct}.
\end{Thm}
\begin{proof}
Let $(t,x)\in \R_+\times \R_+$. If $\tau=T$ is deterministic, we set $\tau_\eps=T-\eps$ for $\eps>0$. In the case where $\tau=\tau_D$ is given by \eqref{eq:def_stop_time}, the stopping time $\tau_\eps$ is defined as in \eqref{eq:def_tau_eps}.

Observe that $Y$ and $Y^L$ satisfy the same dynamics before time $\tau_\eps $. Hence, the results from Lemma \ref{Lemma:mart_property} remain to hold true if $Y^L$ and $X^L$ are replaced by $Y$ and $X^*$. In particular, it follows that the process
$$\theta_s=Y_{s} |X_{s}^*|^{p-1} - Y_{t\wedge \tau_\eps} |X_{t\wedge \tau_\eps}^*|^{p-1}+ \int_{t\wedge \tau_\eps}^{s} \gamma_u |X^*_u|^{p-1} du, \quad s\ge t\wedge \tau_\eps, \ \eps > 0,$$
is a nonnegative local martingale on the stochastic interval $[\![t \wedge \tau_\eps,\tau[\![$ for any $\eps > 0$. Consequently it is a nonnegative supermartingale and thus converges almost surely in $\R$ as $s$ goes to $\tau$ (see Chapter V.3 in \cite{jaco:79} or Appendix A in \cite{carr:fish:ruf:14}). Hence 
$$0 \leq X^*_{s} = \left( \frac{\theta_{s} - p \int_{t\wedge \tau_\eps}^{s} \gamma_u |X^*_u|^{p-1} du}{pY_{s\wedge \tau}} \right)^{q-1} \leq \left( \frac{\theta_{s}}{pY_{s}} \right)^{q-1}.$$
Since $Y$ satisfies the terminal condition $\liminf_{s\nearrow \tau} Y_{s} \1_{\cS} = \infty$ we have a.s. on the set $\{t <\tau\}\cap \cS$:
$$0 \leq X^*_{s}  \leq \left( \frac{\theta_{s}}{pY_{s}} \right)^{q-1} \to 0$$
when $s$ goes $\tau$ . It follows that $X$ satisfies \eqref{eq:terminal_cons} if $t<\tau$.

Appealing once more to Lemma \ref{Lemma:mart_property} we observe that for $t\leq s < \tau $
\begin{eqnarray*}
d(Y_s (X^*_s)^p)& = & -\left[ \eta_s |\al^*_s|^p + \gamma_s (X^*_{s})^{p} \right] ds - \int_\cZ \lambda_s(z) |\beta^*_s(z)|^p \mu(dz) ds\\
&& \quad + (X^*_{s^-})^{p} dM_s + (X^*_{s^-})^{p}   \int_\cZ (Y_{s^-} + \psi_t(z)) \left[ \left( 1-\zeta_s(z) \right)^{p} - 1 \right] \tpi(dz,ds)
\end{eqnarray*}
Since $|X^*_t| \leq x$ we deduce for all $\eps>0$
\begin{eqnarray*}
Y_t |x|^p &= & \1_{\{t<\tau\}}\E^{\F_t} \left[ \int_t^{\tau_\eps \vee t} \left\{\eta_u |\al^*_u|^p + \gamma_u (X^*_{u})^{p} + \int_\cZ \lambda_u(z) |\beta^*_u(z)|^p \mu(dz) \right\} du + Y_{\tau_\eps \vee t} |X_{\tau_\eps \vee t}|^p\right] \\
&&\qquad +\1_{\{t\ge\tau\}}\xi |x|^p\\
& \geq &   \1_{\{t<\tau\}}\E^{\F_t} \left[ \int_t^{\tau_\eps \vee t} \left\{\eta_u |\al^*_u|^p + \gamma_u (X^*_{u})^{p} + \int_\cZ \lambda_u(z) |\beta^*_u(z)|^p \mu(dz) \right\} du + \1_{\{\xi<\infty\}}Y_{\tau_\eps \vee t} |X_{\tau_\eps \vee t}|^p\right]\\
&&\qquad +\1_{\{t\ge\tau\}}J(t,X^*)
\end{eqnarray*}
Appealing to monotone convergence theorem yields
\begin{eqnarray*}
\lim_{\eps\to 0}\1_{\{t<\tau\}}\E^{\F_t} \left[ \int_t^{\tau_\eps \vee t} \left\{\eta_u |\al^*_u|^p + \gamma_u (X^*_{u})^{p} + \int_\cZ \lambda_u(z) |\beta^*_u(z)|^p \mu(dz) \right\} du\right]\\
=\1_{\{t<\tau\}}\E^{\F_t} \left[ \int_t^{\tau} \left\{\eta_u |\al^*_u|^p + \gamma_u (X^*_{u})^{p} + \int_\cZ \lambda_u(z) |\beta^*_u(z)|^p \mu(dz) \right\} du\right]
\end{eqnarray*}
Since we have $\liminf_{\eps\to 0}Y_{\tau_\eps}\ge \xi$ and by Fatou's lemma, we obtain\footnote{Recall that $0\cdot \infty:= 0$}
\begin{eqnarray*}
 \liminf_{\eps\to 0}\1_{\{t<\tau\}}\E^{\F_t}\left[\1_{\{\xi<\infty\}}Y_{\tau_\eps \vee t} |X_{\tau_\eps \vee t}|^p\right] &\ge & \1_{\{t<\tau\}}\E^{\F_t}\left[\liminf_{\eps\to 0}\1_{\{\xi<\infty\}}Y_{\tau_\eps \vee t} |X_{\tau_\eps \vee t}|^p\right]\\
 &\ge & \1_{\{t<\tau\}}\E^{\F_t}\left[\1_{\{\xi<\infty\}}\xi |X_{\tau}|^p\right]\\
  &= & \1_{\{t<\tau\}}\E^{\F_t}\left[\xi |X_{\tau}|^p\right]
 \end{eqnarray*}
Alltogether we obtain that $Y_t |x|^p \geq J(t,X^*).$
Next, note that for every $X \in \cA(t,x)$ we have $J(t,X) \geq J^L(t,X)$. This implies $v(t,x) \geq v^L(t,x)$ for every $L > 0$. By Proposition \ref{prop:optimal_control_uncons} we have $Y^L_t|x|^p = v^L(t,x)$. Minimality of $Y$ implies
$$Y_t|x|^p = \lim_{L\nearrow \infty} Y^L_t|x|^p =  \lim_{L\nearrow \infty} v^L(t,x) \leq v(t,x).$$
Consequently we obtain
$$Y_t|x|^p \geq J(t,X^*) \geq v(t,x) \geq Y_t|x|^p$$
 and thus optimality of $X^*$.
\end{proof}

\section*{Appendix}

\subsection*{Some details concerning the proof of Proposition \ref{thm:exists_sing_sol}}

In this section we give the details for the proof of Proposition \ref{thm:exists_sing_sol}. The constant $\ell$ is defined in Condition \textbf{A6}. Let us begin with two results contained in \cite{krus:popi:14}. For $\zeta \in L^\ell(\Omega)$, let $(Y,\psi,M) \in \bS^\ell(0,T) \times L^\ell_\pi(0,T) \times \cM^\ell(0,T)$ be the classical solution of the BSDE:
$$Y_t = \zeta+ \int_t^T g(u,Y_u,\psi_u) du - \int_t^T \int_\cZ \psi_{u^-}(z) \tpi(dz,du) - \int_t^T dM_u$$
where the generator $g$ satisfies Conditions \textbf{A1}, \textbf{A2} and \textbf{A3} and $g^0_t = g(t,0,0)$ is in $\bH^\ell(0,T)$. Again the existence and the uniqueness of $(Y,\psi,M)$ comes from Theorem 2 in \cite{krus:popi:14}. Recall that $\nu(x) = |x|^{-1} x \1_{x\neq 0}$. The first result is the It\^o formula.
\begin{Lemma}[Corollary 1 and Remark 1 in \cite{krus:popi:14}] \label{lem:appendix_1}
Let $c(\ell) =\frac{\ell((\ell-1)\wedge 1)}{2}$ and $0\leq s \leq t \leq T$, then it holds that
\begin{eqnarray} \nonumber
&& |Y_s|^\ell \leq |Y_t|^\ell + \ell \int_s^t |Y_u|^{\ell-1} \nu(Y_u) g(u,Y_u,\psi_u) du  -c(\ell) \int_s^t  |Y_{u}|^{\ell-2}  \1_{Y_u\neq 0} d[ M ]^c_u  \\ \nonumber
& &\quad  -  \ell  \int_s^t |Y_{u^-}|^{\ell-1} \nu(Y_{u^-}) dM_u -  \ell \int_s^t |Y_{u^-}|^{\ell-1} \nu(Y_{u^-})  \int_\cZ \psi_s(z) \tpi(dz,du)   \\ \nonumber
& &\quad -  \int_s^t  \int_\cZ \left[ |Y_{u^-}+\psi_u(z)|^\ell -|Y_{u^-}|^\ell - \ell|Y_{u^-}|^{\ell-1} \nu(Y_{u^-})  \psi_u(z) \right] \pi(dz,du) \\  \nonumber
&& \quad - \sum_{s < u \leq t}  \left[ |Y_{u^-}+\Delta M_u|^\ell - |Y_{u^-}|^\ell - \ell|Y_{u^-}|^{\ell-1} \nu(Y_{u^-})  \Delta M_u \right] . 
\end{eqnarray}
Moreover 
$\int_0^t \1_{Y_u=0}d[M]^c_u= 0$.
\end{Lemma}
The second result is the following.
\begin{Lemma}[Lemma 9 in \cite{krus:popi:14}]\label{lem:appendix_2}
If $\ell < 2$, the non-decreasing processes involving the jumps of $Y$ control the quadratic variations:
\begin{eqnarray*}
&&\sum_{0< u \leq t}  \left[ |Y_{u^-}+\Delta M_u|^\ell - |Y_{u^-}|^\ell - \ell|Y_{u^-}|^{\ell-1} \nu(Y_{u^-}) \Delta M_u \right]  \\
&& \quad \geq c(\ell)  \sum_{0< u \leq t}|\Delta M_u|^2  \left( |Y_{u^-}|^2 \vee  |Y_{u^-} + \Delta M_u|^2 \right)^{\ell/2-1} \1_{|Y_{u^-}|\vee |Y_{u^-} + \Delta M_u| \neq 0}
\end{eqnarray*}
and
 \begin{eqnarray*}
&& \int_{0}^{t}  \int_\cZ \left[ |Y_{u^-}+\psi_u(z)|^\ell -|Y_{u^-}|^\ell - \ell|Y_{u^-}|^{\ell-1}  \nu(Y_{u^-}) \psi_u(z) \right] \pi(dz,du) \\
&& \quad \geq c(\ell) \int_{0}^{t}  \int_\cZ |\psi_u(z)|^2  \left( |Y_{u^-}|^2 \vee  |Y_{u^-} +\psi_u(z)|^2 \right)^{\ell/2-1} \1_{|Y_{u^-}|\vee |Y_{u^-} + \psi_u(z)| \neq 0} \pi(dz,du).
\end{eqnarray*}
\end{Lemma}

The main step in the proof of Proposition \ref{thm:exists_sing_sol} is the convergence of the solution $(Y^L,\psi^L,M^L)$ of the BSDE \eqref{eq:truncated_bsde} with terminal condition $\xi^L = \xi \wedge L$. In order to carry out this step, we need suitable a priori estimates for the difference $Y^L-Y^N$. We proceed as in Proposition 3 in \cite{krus:popi:14}. These are established in Lemma \ref{lem:appendix_5} below. Let $0 \leq s \leq  t < T$. For $L$ and $N$ nonnegative, we put
$$\what Y_s = Y^{N}_s - Y^{L}_s, \quad \what \psi_s(z) = \psi^{N}_s(z) - \psi^{L}_s(z), \quad \what M_s = M^{N}_s - M^{L}_s.$$
W.l.o.g. we may assume that $\ell \leq 2$ and we choose $a= \ell  \|\vartheta\|^2_{L^2_\mu}/(\ell-1)$. Then It\^o's formula (see Lemma \ref{lem:appendix_1} above) implies
\begin{eqnarray} \nonumber
&& e^{as}|\what Y_s|^\ell \leq e^{at}|\what Y_t|^\ell -\int_s^t a e^{au} |\what Y_u|^\ell du \\ \nonumber
&& \quad + \ell \int_s^t e^{au}|\what Y_u|^{\ell-1} \nu(\what Y_u) (f^N(u,Y^N_u,\psi^N_u) - f^L(u,Y^L_u,\psi^L_u))du   \\ \nonumber
& &\quad  -  \ell  \int_s^t e^{au}|\what Y_{u^-}|^{\ell-1} \nu(\what Y_{u^-}) d\what M_u -  \ell \int_s^t e^{au}|\what Y_{u^-}|^{\ell-1} \nu(\what Y_{u^-})  \int_\cZ \what \psi_(z) \tpi(dz,du)   \\ \nonumber
& &\quad -  \int_{s}^{t}e^{au}  \int_\cZ \left[ |\what Y_{u^-}+\what \psi_u(z)|^\ell -|\what Y_{u^-}|^\ell - \ell|\what Y_{u^-}|^{\ell-1} \nu(\what Y_{u^-})  \psi_u(z) \right] \pi(dz,du) \\  \nonumber
&& \quad - \sum_{0< s \leq t} e^{au} \left[ |\what Y_{u^-}+\Delta \what M_u|^\ell - |\what Y_{u^-}|^\ell - \ell|\what Y_{u^-}|^{\ell-1} \nu(\what Y_{u^-})  \Delta \what M_u \right] \\ \label{eq:Ito_formula_ell}
&&\quad -c(\ell) \int_s^t e^{au} |\what Y_{u}|^{\ell-2}  \1_{\what Y_u\neq 0} d[ \what M ]^c_u .
\end{eqnarray}
Here $\nu (x) = |x|^{-1} x \1_{x\neq0}$ and $c(\ell)= \ell(\ell-1)/2$. For the term containing the generators we have
\begin{eqnarray*}
&& |\what Y_u|^{\ell-1} \nu(\what Y_u) (f^N(u,Y^N_u,\psi^N_u) - f^L(u,Y^L_u,\psi^L_u))\\
&&\quad  \leq  |\what Y_u|^{\ell-1} \nu(\what Y_u) (f^N(u,Y^N_u,\psi^N_u) - f^L(u,Y^N_u,\psi^N_u)) \\
&&\qquad + |\what Y_u|^{\ell-1} \nu(\what Y_u) (f^L(u,Y^N_u,\psi^N_u) - f^L(u,Y^L_u,\psi^L_u))   \\
&& \quad \leq |\what Y_u|^{\ell-1} \nu(\what Y_u) (f^0_u \wedge N- f^0_u \wedge L) +  |\what Y_u|^{\ell-1} \nu(\what Y_u) (f^L(u,Y^N_u,\psi^N_u) - f^L(u,Y^N_u,\psi^L_u))\\
& & \quad \leq |\what Y_u|^{\ell-1} |f^0_u \wedge N- f^0_u \wedge L| +  |\what Y_u|^{\ell-1} \left| \int_\cZ \what \psi_u(z) \kappa^{Y^N,\psi^N,\psi^L}_u(z) \mu(dz) \right| \\
&& \quad \leq |\what Y_u|^{\ell-1} |f^0_u \wedge N- f^0_u \wedge L| + \|\vartheta\|_{L^2_\mu} |\what Y_u|^{\ell-1} \|\what \psi_u\|_{L^2_\mu} 
\end{eqnarray*}
where we used monotonicity \textbf{A1} of $f^{L}$ w.r.t. $y$ (with $\chi=0$) and the condition \textbf{A2} of $f^L$ w.r.t. $\psi$. Then by Young's inequality
$$\ell  \|\vartheta\|_{L^2_\mu} |\what Y_u|^{\ell-1} \|\what \psi_u\|_{L^2_\mu} \leq \frac{\ell}{(\ell-1)}  \|\vartheta\|^2_{L^2_\mu}|\what Y_u|^{\ell} + \frac{c(\ell)}{2}|\what Y_u|^{\ell-2}  \|\what \psi_u\|^2_{L^2_\mu}.$$
We define
$$X =e^{at}  |\what Y_t|^\ell + \ell \int_0^t e^{au} |\what Y_u|^{\ell-1} |f^0_u \wedge N-f^0_u \wedge L| du .$$
From Lemma \ref{lem:appendix_2} we obtain for every $s \in [0,t]$:
\begin{eqnarray} \nonumber
&&e^{as} |\what Y_s|^\ell +c(\ell) \sum_{s< u \leq t}e^{au}|\Delta \what M_u|^2  \left( |\what Y_{u^-}|^2 \vee  |\what Y_{u^-} + \Delta \what M_u|^2 \right)^{\ell/2-1} \1_{|\what Y_{u^-}|\vee |\what Y_{u^-} + \Delta \what M_u| \neq 0} \\ \nonumber
&& \quad + c(\ell) \int_{s}^{t} e^{au} \int_\cZ |\what \psi_u(z)|^2  \left( |\what Y_{u^-}|^2 \vee  |\what Y_{u^-} +\what \psi_u(z)|^2 \right)^{\ell/2-1} \1_{|\what Y_{u^-}|\vee |\what Y_{u^-} + \what \psi_u(z)| \neq 0} \pi(dz,du) \\ \nonumber
&& \quad + c(\ell) \int_s^t e^{au} |\what Y_{u}|^{\ell-2}  \1_{\what Y_u\neq 0} d[ \what M ]^c_u -\frac{c(\ell)}{2} \int_s^t |\what Y_u|^{\ell-2}  \|\what \psi_u\|^2_{L^2_\mu} du \\  \label{eq:Ito_conv}
&& \leq X -  \ell  \int_s^t e^{au}|\what Y_{u^-}|^{\ell-1} \nu(\what Y_{u^-}) d\what M_u -  \ell \int_s^t e^{au} |\what Y_{u^-}|^{\ell-1} \nu(\what Y_{u^-})  \int_\cZ \what \psi_u(z) \tpi(dz,du) .
\end{eqnarray}
Indeed from the choice of $a$, the terms
$$\frac{\ell  \|\vartheta\|^2_{L^2_\mu}}{\ell-1}  \int_s^t e^{au} |\what Y_u|^\ell du = a \int_s^t e^{au} |\what Y_u|^\ell du$$
cancel each other. 

Lemma \ref{lem:appendix_5} is a consequence of the following two lemmas. 
\begin{Lemma}\label{lem:appendix_3}
There exists a constant $C_\ell$ depending only on $\ell$ such that for any $0 < t < T$
\begin{equation}\label{eq:conv_S_ell_Y}
\E \left(\sup_{s\in [0,t]}e^{as} |\what Y_s|^\ell \right) \leq C_\ell \E (X).
\end{equation}
\end{Lemma}
\begin{proof}
Indeed we take $\tau_k$ as a fundamental sequence of stopping times for the local martingale 
$$\int_{0}^{.} e^{au}|\what Y_{u^-}|^{\ell-1} \nu(\what Y_{u^-}) \left( d\what M_u  +  \int_\cZ \what \psi_u(z) \tpi(dz,du) \right)$$
and $\hat \tau_k$ as a localization time
$$\hat \tau_k  = \inf \left\{ t\geq 0, \ \int_{0}^{t}  \int_\cZ e^{au} |\what \psi_u(z)|^2  \left( |\what Y_{u^-}|^2 \vee  |\what Y_u|^2 \right)^{\ell/2-1} \1_{|\what Y_{u^-}|\vee |\what Y_u| \neq 0} \pi(dz,du) \geq k \right\}.$$
We set $\tau = \tau_k \wedge \hat \tau_k\wedge t$. Now we have:
\begin{eqnarray*}\nonumber
&& \E  \int_{0}^{\tau} e^{au}  \int_\cU |\what \psi_s(u)|^2  \left( |\what Y_{s^-}|^2 \vee  |\what Y_{s^-} +\what \psi_s(u)|^2 \right)^{p/2-1} \1_{|\what Y_{s^-}|\vee |\what Y_{s^-} + \what \psi_s(u)| \neq 0} \pi(du,ds) \\ \nonumber
&& \quad =  \E  \int_{0}^{\tau} e^{au} \int_\cU |\what \psi_s(u)|^2  \left( |\what Y_{s^-}|^2 \vee  |\what Y_{s}|^2 \right)^{p/2-1} \1_{|\what Y_{s^-}|\vee |\what Y_{s}| \neq 0} \pi(du, ds) \\ 
&& \quad = \E  \int_{0}^{\tau}e^{au}  \int_\cU |\what \psi_s(u)|^2  |\what Y_{s}|^{p-2} \1_{\what Y_{s} \neq 0} \mu(du) ds  = \E  \int_{0}^{\tau} e^{au}\|\what \psi_s\|_{L^2_\mu}^2  |\what Y_{s}|^{p-2} \1_{\what Y_{s} \neq 0} ds. 
\end{eqnarray*}
From this equality and taking the expectation in \eqref{eq:Ito_conv} we deduce that 
\begin{eqnarray} \nonumber
&&c(\ell) \E  \sum_{0< u \leq \tau}e^{au}|\Delta \what M_u|^2  \left( |\what Y_{u^-}|^2 \vee  |\what Y_{u^-} + \Delta \what M_u|^2 \right)^{\ell/2-1} \1_{|\what Y_{u^-}|\vee |\what Y_{u^-} + \Delta \what M_u| \neq 0} \\ \nonumber
&& \quad + c(\ell) \E  \int_0^\tau e^{au} |\what Y_{u}|^{\ell-2}  \1_{\what Y_u\neq 0} d[ \what M ]^c_u + \frac{c(\ell)}{2} \E \int_0^\tau e^{au} |\what Y_u|^{\ell-2}  \|\what \psi_u\|^2_{L^2_\mu} du \\ \nonumber
&& \quad + \frac{c(\ell)}{2}\E \int_{0}^{\tau} e^{au} \int_\cZ |\what \psi_u(z)|^2  \left( |\what Y_{u^-}|^2 \vee  |\what Y_{u^-} +\what \psi_u(z)|^2 \right)^{\ell/2-1} \1_{|\what Y_{u^-}|\vee |\what Y_{u^-} + \what \psi_u(z)| \neq 0} \pi(dz,du) \\ \label{eq:Lp_estim_p_leq_2} 
&&\quad \leq 2 \E (X) 
\end{eqnarray}
and we can allow $\tau$ to be equal to $t$ in this last inequality. Then using the Burkholder-Davis-Gundy inequality  in \eqref{eq:Ito_conv} we obtain that:
$$\E\left(  \sup_{0\leq s \leq t} e^{as} |\what Y_s|^\ell\right) \leq \E(X) +k_\ell \E\left(  [M^Y]_t^{1/2} +  [\tpi^Y]_t^{1/2} \right)$$
with
$$M^Y_s + \tpi^Y_s = \ell  \int_0^s e^{au}|\what Y_{u^-}|^{\ell-1} \nu(\what Y_{u^-}) d\what M_u +  \ell \int_0^s e^{au} |\what Y_{u^-}|^{\ell-1} \nu(\what Y_{u^-})  \int_\cZ \what \psi_u(z) \tpi(dz,du).$$
Since $\ell>1$, the bracket of the first martingale is controlled by:
\begin{eqnarray*}
&&k_\ell  \E\left(  [M^Y]_t^{1/2} \right) \\
&& \leq k_\ell \E\left[ \left( \int_0^t e^{2au} \left( |\what Y_{u^-}|^2 \vee  |\what Y_{u^-} + \Delta \what M_u|^2 \right)^{\ell-1}  \1_{|\what Y_{u^-}|\vee |\what Y_{u^-} + \Delta \what M_u| \neq 0} d[\what M]_u \right)^{1/2} \right] \\
&& \quad \leq  \frac{1}{4}  \E \left( \sup_{0\leq u \leq t} e^{au} |\what Y_u|^\ell \right) + k_\ell^2 \E \left( \int_0^T  e^{au} |\what Y_{u^-}|^{\ell-2}  \1_{|\what Y_{u^-}| \neq 0} d[\what M]^c_u \right) \\
&& \qquad \qquad + k_\ell^2 \E \left( \sum_{0< s \leq T} e^{au} \left( |\what Y_{u^-}|^2 \vee  |\what Y_{u^-} + \Delta \what M_u|^2 \right)^{\ell/2-1} \1_{|\what Y_{u^-}|\vee |\what Y_{u^-} + \Delta \what M_u| \neq 0} |\Delta \what M_u|^2 \right)
\end{eqnarray*}
and for the second
\begin{eqnarray*}
k_\ell \E \left(  [\tpi^Y]_t^{1/2} \right) & \leq & k_\ell \E \left[\left( \sup_{0\leq u \leq t} \left( e^{au} |\what Y_u|^\ell \right)\right)^{\frac{1}{2}} \left( \int_0^t e^{au} |Y_{u}|^{\ell-2}\1_{Y_u\neq 0} \int_\cZ |\psi_u(z)|^2 \pi(dz,ds) \right)^{\frac{1}{2}} \right] \\
& \leq  & \frac{1}{4}  \E \left( \sup_{0\leq u \leq t} e^{au} |\what Y_u|^\ell  \right) +  k_p^2 \E \left( \int_0^t e^{au} |Y_{u}|^{\ell-2} \|\psi_u\|^2_{L^2_\mu} \1_{Y_u \neq 0} du \right).
\end{eqnarray*}
Hence the Inequality \eqref{eq:conv_S_ell_Y} is proved. 
\end{proof}

We apply again Young's inequality to obtain that
\begin{equation}\label{eq:appendix_2}
C_\ell \E(X) \leq C_\ell \E \left( e^{at} |\what Y_t|^\ell \right) + \frac{1}{2} \E \left(\sup_{s\in [0,t]} e^{as} |\what Y_s|^\ell \right) + \bar C_\ell \E \int_0^t e^{au} |f^0_u \wedge N-f^0_u \wedge L|^\ell du
\end{equation}
and we can conclude that 
\begin{equation} \label{eq:appendix_1}
\E \left(\sup_{s\in [0,t]} e^{as} |\what Y_s|^\ell \right) \leq \hat C_\ell \E \left( e^{at}|\what Y_t|^\ell \right) + \hat C_\ell \E \int_0^t e^{au} |f^0_u \wedge N-f^0_u \wedge L|^\ell du.
\end{equation}
Next, we derive a similar inequality for $\psi^L$ and $M^L$. 
\begin{Lemma}\label{lem:appendix_4}
There exists a constant $\wtil C_\ell$ such that for any $0 < t < T$
\begin{equation*}
\E \left[ \left( \int_0^t e^{2as/\ell} d[\what M]_s \right)^{\ell/2} + \left( \int_0^t e^{2as/\ell}  \int_\cZ |\psi_s(z)|^2 \mu(dz)ds \right)^{\ell/2}\right] \leq \wtil C_\ell \E(X).
\end{equation*}
\end{Lemma}
\begin{proof}
From Lemma \ref{lem:appendix_1}, it holds a.s.
$$\int_0^t \1_{\what Y_s=0} d[\what M]^c_s = 0.$$
Hence
\begin{eqnarray} \nonumber
&& \E \left[ \left( \int_0^t e^{2as/\ell} d[\what M]^c_s \right)^{\ell/2} \right] =  \E\left[  \left( \int_0^t  e^{2as/\ell} \1_{Y_s \neq 0} d[\what M]^c_s\right)^{\ell/2}\right] \\ \nonumber
&& \quad \leq \E \left[ \left( \sup_{0\leq u\leq t}e^{au} |\what Y_u|^\ell\right)^{(2-\ell)/2} \left(\int_0^t e^{as} \left| \what Y_{s} \right|^{\ell-2}\1_{\what Y_s \neq 0}  d[\what M]^c_s \right)^{\ell/2}\right]\\ \nonumber
&& \quad  \leq \frac{2-\ell}{2} \E \left[ \sup_{0\leq u\leq t} e^{au} |\what Y_u|^\ell\right] + \frac{\ell}{2} \E \int_0^t e^{as} \left| \what Y_{s} \right|^{\ell-2} \1_{\what Y_s \neq 0}  d[\what M]^c_s 
\end{eqnarray}
where we have used H\"older's and Young's inequality with $ \frac{2-\ell}{2} + \frac{\ell}{2}=1$. With Inequality \eqref{eq:Lp_estim_p_leq_2} we deduce:
\begin{eqnarray*}
&&  \E \left[ \left( \int_0^t e^{2as/\ell} d[\what M]^c_s \right)^{\ell/2} \right] \leq \wtil C_p \E(X).
\end{eqnarray*}
For the pure-jump part of $[M]$, let $\eps > 0$ and consider the function $u_\eps(y) = (|y|^2 + \eps^2)^{1/2}$. Then
\begin{eqnarray*}
&& \E \left[\left( \sum_{0<s\leq t} e^{2as/\ell} |\Delta \what M_s|^2 \right)^{\ell/2} \right]\\
&& \quad \leq \E \left[ \left( \sup_{0\leq s\leq t} e^{as/\ell} u_\eps(\what Y_{s}) \right)^{\ell(2-\ell)/2} \left(\sum_{0<s\leq t} e^{as} \left( u_\eps(|\what Y_{s^-}| \vee  |\what Y_{s^-} + \Delta \what M_s|) \right)^{\ell-2} |\Delta \what M_s|^2 \right)^{\ell/2}\right]\\
&& \quad \leq \left\{\E \left[ \left( \sup_{0\leq s\leq t} e^{as/\ell}u_\eps(\what Y_{s})\right)^{\ell}\right]\right\}^{(2-\ell)/2} \\
&&\qquad \qquad \times \left\{\E\left(  \sum_{0< s \leq t} e^{as} \left( u_\eps(|\what Y_{s^-}| \vee  |\what Y_{s^-} + \Delta \what M_s|) \right)^{\ell-2} |\Delta \what M_s|^2 \right) \right\}^{\ell/2} \\
&& \quad \leq \frac{2-\ell}{2} \E \left[  \sup_{0\leq s\leq t} e^{as} u_\eps(\what Y_{s})^\ell \right] + \frac{\ell}{2}\E\left(  \sum_{0< s \leq t}e^{as} \left( u_\eps(|\what Y_{s^-}| \vee  |\what Y_{s^-} + \Delta \what M_s| )\right)^{\ell-2} |\Delta \what M_s|^2 \right)
\end{eqnarray*}
Let $\eps$ go to zero with Inequality \eqref{eq:Lp_estim_p_leq_2}
\begin{eqnarray*}
&& \E\left[ \left(  \sum_{0<s\leq t} e^{2as/\ell}|\Delta \what M_s|^2  \right)^{\ell/2}\right]  \leq \frac{2-\ell}{2} \E \left(  \sup_{0\leq s\leq t} e^{as}|\what Y_{s}|^\ell \right) \\
&&\qquad + \frac{\ell}{2}\E\left( \sum_{0 <s\leq t} e^{as} \left( |\what Y_{s^-}| \vee  |\what Y_{s^-} + \Delta \what M_s| \right)^{\ell-2}  \1_{|\what Y_{s^-}|\vee |\what Y_{s^-} + \Delta \what M_s| \neq 0}|\Delta \what M_s|^2 \right)\\
&& \quad \leq  \wtil C_\ell \E(X).
\end{eqnarray*}
The same argument shows that 
\begin{eqnarray*}
\E \left[ \left( \int_0^te^{2as/\ell} \int_\cZ |\psi_s(z)|^2 \mu(dz)ds \right)^{\ell/2} \right] & \leq &\wtil C_\ell \E(X). 
\end{eqnarray*}
\end{proof}

Combining estimates of Lemmas \ref{lem:appendix_3} and \ref{lem:appendix_4} with Inequalities \eqref{eq:appendix_2} and \eqref{eq:appendix_1} we obtain the desired result:
\begin{Lemma} \label{lem:appendix_5}
There exists a constant $K_\ell$ such that for any $0 < t < T$
\begin{eqnarray*}
&& \E\left[ \sup_{s\in [0,t]} e^{as} |\what Y_s|^\ell +  \left(  \int_0^te^{2as/\ell} \int_\cZ |\what \psi_u(z)|^2\mu(dz) du \right)^{\ell/2} +\left( \int_0^t e^{2as/\ell}d[\what M]_s \right)^{\ell/2} \right]\\
&&\qquad \leq K_{\ell} \E \left(e^{at} |\what Y_t|^\ell \right) + K_{\ell} \E \left( \int_0^t e^{au}|f^0_u \wedge N- f^0_u \wedge L|^\ell du\right)
\end{eqnarray*}
where $K_{\ell}$ depends only on $\ell$.
\end{Lemma}

\subsection*{Some details concerning the conditions \textbf{B} and \textbf{A3'}}

Recall that $\delta^*$ and $h^*$ are defined by the formulas \eqref{eq:def_delta_min} and \eqref{eq:def_h_min}.
\begin{Lemma}\label{lem:optimal_delta_h}
If $\rho > \delta^*$ and $m > h^*$, then there exists $r >1$ such that 
$$r \left[ \chi + \frac{K^2}{2((r-1)\wedge 1)} \right] < \rho \quad \mbox{and} \quad \frac{r\delta}{\rho - \delta} < m.$$
\end{Lemma}
\begin{proof}
Let us define the function $\delta\colon (1,\infty) \to \R$, 
$$\delta(r) =  r \left[ \chi + \frac{K^2}{2((r-1)\wedge 1)} \right].$$
We show that $\delta^*$ is the minimal value of $\delta$. We first assume that $K \neq 0$. Then $\lim_{r \to 1} \delta(r) = + \infty$. 

\begin{itemize}
\item \textbf{Case 1:} $\chi < -K^2/2$. $\delta$ is decreasing and tends to $-\infty$ as $r$ tends to $+\infty$. Thus $\delta^* = -\infty$.

\item \textbf{Case 2:} $\chi = -K^2/2$. $\delta$ is a non increasing function with $\delta(r) > 0$ for any $r < 2=r^*$ and $\delta(r) = 0$ for any $r \geq 2=r^*$. Hence $\delta^* = 0$.

\item \textbf{Case 3:} $\chi > - K^2/2$. The function $\delta$ tends to $+\infty$ when $r$ tends to $+\infty$ and has a strict minimum at $r^* \in [1,2]$:
$$r^* = 1 + \left( \mathbf{1}_{ - \frac{K^2}2<\chi \leq \frac{K^2}{2}} + \frac{K}{\sqrt{2\chi}} \mathbf{1}_{ \chi > \frac{K^2}{2}} \right).$$
Moreover the minimum $\delta^*=\delta(r^*) > 0$ is given by:

$$\delta^*= \left\{ \begin{array}{ll}
2\left( \chi + \frac{K^2}{2} \right) = K^2 + 2\chi & \mbox{if } -K^2<2\chi \leq K^2,\\
\chi \left( 1+ \frac{K}{\sqrt{2\chi}}\right)^2 = \chi( r^*)^2 & \mbox{if } 2\chi > K^2.
\end{array} \right.$$ 
\end{itemize}
Gathering together the above results implies that $\delta^*$ defined in Equation \eqref{eq:def_delta_min} is the minimal value of $\delta$.

Therefore if $\rho > \delta^*$ (Condition \textbf{(B)}), there exists an open interval $(R_1,R_2)$ such that for any $r \in (R_1,R_2)$, $\rho > \delta(r) \geq \delta^*$. In Case 1, we have $1 < R_1$ and $R_2 = +\infty$; in Case 2, $1<R_1 < 2$ and $R_2 = +\infty$, and in Case 3, $1 < R_1 < r^* < R_2 < +\infty$. Let us define on $(R_1,R_2)$ the function
$$h(r) = \frac{\rho r}{\rho-\delta(r)}.$$
\begin{itemize}
\item \textbf{Case 1:} here $R_2= +\infty$, $\delta^*= -\infty$. The optimal choice of $\rho$ is $\rho < 0$ (see Remark \ref{rem:chi_small}). Then for any $r\in (R_1,+\infty)$, $h(r) \leq 0 <m$. 
\end{itemize}
In the other cases we will prove that the minimum value of $h$ on $(R_1,R_2)$ is $h^*$. Hence if $m> h^*$ (Condition \textbf{A3'}), there exists a value $r \in (R_1,R_2)$ such that $m > h(r) \geq h^*$ and since $\rho > \delta(r)$ on this interval, the lemma is proved.

Note that $\lim_{r \to R_1} h(r) = +\infty$ and $\rho > 0$ since $\delta^*\geq 0$. The derivative of $h$ (expect for $r=2$) is equal to
$$h'(r) = \frac{\rho}{(\rho-\delta(r))^2} \left( \rho-\delta(r)+r\delta'(r)\right).$$
For $r > 2$, $h'(r) = \rho^2 / (\rho-\delta(r))^2 > 0$. For $1 < r < 2$, we have
\begin{eqnarray*}
h'(r) & = & \frac{\rho}{(\rho-\delta(r))^2} \left( \rho - \frac{K^2}{2} \frac{r^2}{(r-1)^2} \right) \\
& = &  \frac{\rho}{(\rho-\delta(r))^2}\left(\sqrt{ \rho} - \frac{Kr}{\sqrt{2}(r-1)} \right)\left(\sqrt{ \rho} + \frac{Kr}{\sqrt{2}(r-1)} \right) .
\end{eqnarray*}
Therefore for some $r^{\dag} \in (1,2)$, $h'(r^{\dag})=0$ if and only if :
$$\frac{\sqrt{2 \rho}}{K} = \frac{r^{\dag}}{r^{\dag}-1} \Leftrightarrow \rho > 2K^2 \mbox{ and } r^{\dag} = 1 + \frac{K}{\sqrt{2\rho} - K} \in (1,2).$$
From the convexity of $\delta$ if $r^{\dag}$ exists, then $R_1 < r^\dag < R_2$ and 
$$h(r^{\dag}) = - \frac{\rho}{\delta'(r^{\dag})} = \frac{2\rho}{(\sqrt{2\rho} - K)^2- 2 \chi}.$$
\begin{itemize}

\item \textbf{Case 2:} here $\chi + K^2/2=0$, $R_2= +\infty$. If $\rho \leq 2K^2$ the minimal value of $h$ is attained at $r=2$, with $h^*=2$. If $\rho > 2K^2$, then 
\begin{eqnarray*}
h^* & = & h(r^\dag) = \frac{2\rho}{(\sqrt{2\rho} - K)^2- 2 \chi} = \frac{2\rho}{(\sqrt{2\rho} - K)^2+K^2} = \frac{2\rho}{\rho+ (\sqrt{\rho} - K\sqrt{2})^2} .
\end{eqnarray*}

\item \textbf{Case 3:} here $\rho > \delta^*> 0$ and $1 < R_1 < R_2 < +\infty$.
\begin{itemize}
\item[\textbf{a.}] $\chi < K^2/2$: then $R_2 > 2$. If $\delta^* = K^2 + 2\chi< \rho < 2K^2$, then  
$$h^*=h(2)= \frac{2\rho}{\rho - (K^2 + 2\chi)}.$$
Else if $\rho > 2K^2$ then 
\begin{eqnarray*}
h^*& =& h(r^\dag)=  \frac{2\rho}{(\sqrt{2\rho} - K)^2- 2 \chi} =  \frac{2\rho}{(\sqrt{2\rho} - K)^2 +K^2 - (K^2+ 2 \chi)}\\
& = & \frac{2\rho}{\rho+ (\sqrt{\rho} - K\sqrt{2})^2- (K^2+ 2 \chi)}
\end{eqnarray*}
Finally
$$h^*=  \frac{2\rho}{\rho- (K^2+ 2 \chi)+ (\sqrt{\rho} - K\sqrt{2})^2 \mathbf{1}_{\rho > 2K^2}}.$$
\item[\textbf{b.}] $ \chi \geq K^2/2$. Then $\delta^* \geq 2K^2$. Hence $\rho > 2K^2$. Thus the minimum of $h$ is attained at $h(r^\dag)$ :
\begin{eqnarray*}
h^* & = & h(r^\dag) = \frac{2\rho}{(\sqrt{2\rho} - K)^2- 2 \chi} = \frac{\rho}{\left( \sqrt{\rho} -\sqrt{\chi} - \frac{K}{\sqrt{2}}\right)\left( \sqrt{\rho} +\sqrt{\chi} - \frac{K}{\sqrt{2}}\right)} \\
& = & \frac{\rho}{ \sqrt{\rho} +\sqrt{\chi} - \frac{K}{\sqrt{2}}} \times \frac{1}{\sqrt{\rho} - \left(\sqrt{\chi} + \frac{K}{\sqrt{2}}\right)}.
\end{eqnarray*}

\end{itemize}

\end{itemize}
Let us now summarize the results. $h^*$ is given by (see also Equation \eqref{eq:def_h_min}):
$$h^*= \left\{ \begin{array}{ll}
0 & \mbox{if } 2\chi < - K^2,\\
\frac{2\rho}{\rho- \delta^*+ (\sqrt{\rho} - K\sqrt{2})^2 \mathbf{1}_{\rho > 2K^2}} & \mbox{if } 2|\chi| \leq K^2,\\
 \frac{\rho}{ \sqrt{\rho} +\sqrt{\chi} - \frac{K}{\sqrt{2}}} \times \frac{1}{\sqrt{\rho} - \sqrt{\delta^*}}& \mbox{if } 2\chi > K^2.
\end{array} \right.$$ 
Note for $K=0$ that the formula \eqref{eq:def_delta_min} still holds and for $\chi =0$, $h^*=1$ and for $\chi > 0$, $h^* = \rho/(\rho-\chi)$. 
\end{proof}

\subsection*{Acknowledgements.} 
The authors would like to thank the referees for helpful comments and suggestions. Thomas Kruse acknowledges the financial support from the French Banking Federation through the Chaire "Markets in Transition".

\bibliography{biblio}

\end{document}